\documentclass[11pt,a4paper]{amsart}

\usepackage{graphicx}
\usepackage[inline]{enumitem}
\usepackage{amsmath, amssymb, amsthm, amscd}
\usepackage{xcolor}
\usepackage{caption}
\usepackage{subcaption}
\usepackage{booktabs}
\usepackage{multirow}
\usepackage{stmaryrd}
\usepackage{accents}
\usepackage{bm}
\usepackage{amsmath}
\usepackage{epsfig}
\usepackage{siunitx} 
\usepackage{color,graphics}
\usepackage{booktabs}
\usepackage{caption}
\usepackage{subcaption}
\usepackage{amsthm}
\usepackage{amsmath}
\usepackage{amssymb}
\usepackage{lipsum}
\usepackage{cite}
\usepackage{pdflscape,longtable,array}
\usepackage{afterpage}
\usepackage{subcaption}

\theoremstyle{plain}
\newtheorem{theorem}{Theorem}[section]

\theoremstyle{definition}

\newtheorem{example}[theorem]{Example}

\theoremstyle{remark}
\newtheorem{remark}[theorem]{Remark}
\numberwithin{equation}{section}

\newcommand{\Norm}[1]{{\left\vert\kern-0.25ex\left\vert\kern-0.25ex\left\vert #1 
    \right\vert\kern-0.25ex\right\vert\kern-0.25ex\right\vert}}

\usepackage{fixmath, bm, amssymb, amsfonts, latexsym}
\begin{document}
\author{Lavanya V Salian$^{\S}$, Samala Rathan$^{\S}$ and Rakesh Kumar$^\dagger$}
\title{Compact finite-difference scheme for some Sobolev type equations with Dirichlet boundary conditions}
\thanks{
$^{\S}$Department of Humanities and Sciences, Indian Institute of Petroleum and Energy-Visakhapatnam,
India-530003. (lavanya\_vs@iipe.ac.in,\,rathans.math@iipe.ac.in)
\newline
$^\dagger$School of Data Science, 
Indian Institute of Science Education and Research
Thiruvananthapuram, 
Thiruvananthapuram, India-695551, \& Department of Mathematics, École Centrale School of Engineering, Mahindra University, Hyderabad, 500043, Telangana, India.(rakeshmath21@iisertvm.ac.in, rakesh.kumar@mahindrauniversity.edu.in)}
\date{\today}
 \maketitle
\begin{abstract}
This study aims to construct a stable, high-order compact finite difference method for solving Sobolev-type equations with Dirichlet boundary conditions. Approximation of higher-order mixed derivatives in some specific Sobolev-type equations requires a bigger stencil information.  One can approximate such derivatives on compact stencils, which are higher-order accurate and take less stencil information but are implicit and sparse. Spatial derivatives in this work are approximated using the sixth-order compact finite difference method (Compact6), while temporal derivatives are handled with the explicit forward Euler difference scheme. We examine the accuracy and convergence behavior of the proposed scheme.  Using the von Neumann stability analysis, we establish $L_2-$stability theory for the linear case. We derive conditions under which fully discrete schemes are stable. Also, the amplification factor $\mathcal{C}(\theta)$ is analyzed to ensure the decay property over time. Real parts of $\mathcal{C}(\theta)$ lying on the negative real axis confirm the exponential decay of the solution. A series of numerical experiments were performed to verify the effectiveness of the proposed scheme. These tests include both one-dimensional and two-dimensional cases of advection-free and advection-diffusion flows. They also cover applications to the equal width equation, such as the propagation of a single solitary wave, interactions between two and three solitary waves, undular bore formation, and the Benjamin–Bona–Mahony–Burgers equation.
\end{abstract}

\section{\bf Introduction} 
\label{sec:1}
Sobolev-type equations are partial differential equations (PDEs) that feature third-order mixed time and space derivatives. They describe wave motion in media with nonlinear wave steepening, dispersion, and diffusion, making them essential for various scientific and engineering applications. These equations are widely used to model physical phenomena such as moisture migration in soil \cite{shi1990initial}, fluid flow through fractured rock \cite{barenblatt1960basic}, shear in second-order fluids \cite{ting1974cooling}, consolidation of clay \cite{taylor1942research}, and processes in semiconductors \cite{gao2017weak}, providing valuable insights across disciplines including hydrodynamics \cite{gao2015modified}  and thermodynamics \cite{ewing1977coupled}.
\par
In this paper, we focus on a high-order compact finite difference scheme for solving a class of Sobolev-type equations with a Burgers-type nonlinear term, given by
\begin{equation} 
u_t +  f(u)_x -\gamma u_{xx}  - \delta u_{xxt} = \mathrm{g}(x,t), \quad x \in \Omega \subseteq \mathbb{R} , \quad t \in (0, T],
\label{eqn:A}
\end{equation}
with initial condition
\begin{equation} \label{eqn:A_IC}
u(x,0) = \omega (x), \quad x \in \Omega,
\end{equation}
and Dirichlet boundary conditions. Here, $f(u)$ represents the nonlinear function, $\mathrm{g}$ is a source term, and $\gamma,\,\, \delta \geq 0$ are constants. Equation (\ref{eqn:A}) represents various classes of linear and nonlinear PDEs. The study will primarily concentrate on the Equal Width (EW) and the Benjamin–Bona–Mahony–Burgers (BBMB) equations.

\par The EW equation, introduced by Morrison et al. \cite{morrison1984scattering}, is a mathematical model that describes one-dimensional wave propagation within nonlinear media, containing dispersive effects, is given by
\begin{equation}\label{eqn:EW}
       u_t + u u_x  = \delta u_{xxt}, \quad x \in \Omega, \quad t \in (0, T].
\end{equation}
It is presented as an alternative to both the Regularized Long Wave (RLW) and Korteweg-de Vries (KdV) equations, which are also known to model nonlinear dispersive wave phenomena featuring solitary waves that preserve their shape and speed post-interaction. To simulate physical boundary conditions, the EW equation is equipped with the condition $u = 0 $ at $ x = x_0 $ and $ x = x_N $, signifying that $ u \to 0 $ as $ x \to \pm \infty $, with boundaries positioned sufficiently far apart initially. This PDE is termed the equal width equation due to the fact that, given a specific value of the parameter $ \delta $, solitary wave solutions exhibit a consistent width or wavelength across different wave amplitudes. A bore arises when a deeper flow of water moves into a still-water region, particularly when the depth transition has a gentle slope. Experimental findings suggest that if the depth change ratio is below 0.28, the resulting bore exhibits undular characteristics. However, when this ratio exceeds 0.28, at least one wave crest in the bore begins to break \cite{peregrine1966calculations}. Analytical solutions to the EW equation exist but are confined to specific initial and boundary conditions \cite{hamdi2003exact}. However, several numerical methods have been developed to study it, including those found in \cite{gardner1997simulations, esen2005numerical, saka2006finite, kumar2016b, kumar2020modified, salian2024exponential, salian2024novel}.
\par
The EW equation omits dissipation, but it is essential to examine how numerical methods handle dissipative effects alongside nonlinearity and dispersion. Thus, one can consider the inhomogeneous BBMB equation:
\begin{equation}\label{eqn:BBMB}
    u_t + (1+u) u_x - \gamma u_{xx} - \delta u_{xxt} = \mathrm{g}(x,t), \quad x \in \Omega, \quad t \in (0, T],
\end{equation}
which incorporates dissipation through the term \( \gamma u_{xx} \). References \cite{bruzon2014weak, bruzon2012conservation, bruzon2016conservation} discuss exact solutions and conservation laws for the BBMB equation. The BBM equation is notable for modeling long-wavelength waves in media like fluids and plasma \cite{abbasbandy2010first}. Several analytical methods, such as the tanh-coth \cite{salas2010new}, Exp-function \cite{ganji2009approximate, noor2011some}, and $ G^{\prime}/G $-expansion techniques \cite{bahrami2011exact}, have been employed to address generalized BBM and BBM-Burgers equations. Research also includes investigations of solution stability, decay rates \cite{qinghua2012degenerate, guo2012optimal, yin2010exponential}, and shock profile behavior in two-dimensional cases \cite{xiao2013nonlinear}. Numerical methods, including homotopy perturbation, variational iteration \cite{tari2007approximate}, meshless method of radial basis functions \cite{dehghan2014numerical}, B-spline methods \cite{kumar2016b, joshi2023numerical, jena2023numerical} and finite difference schemes \cite{achouri2006convergence}, have been used to obtain approximate solutions with verified accuracy.\par

\par The objective of this study is to construct a stable, high-order compact finite difference method for solving Sobolev-type equations of the form (\ref{eqn:A}) with  Dirichlet boundary conditions. Approximation of higher-order mixed derivatives in some specific Sobolev-type equations requires a bigger stencil information.  One can approximate such derivatives on compact stencils, which are higher-order accurate and take less stencil information but are implicit and sparse. Spatial derivatives in (\ref{eqn:A}) in this work are approximated using the sixth-order compact finite difference method (Compact6), while temporal derivatives are handled with the explicit forward Euler difference scheme. We examine the accuracy and convergence behavior of the proposed scheme. Using the von Neumann stability analysis, we establish $L_2-$stability theory for the linear case. We derive conditions under which fully discrete schemes ( Euler forward time-stepping combined with Compact6 spatial discretization) are stable. Also, the amplification factor $\mathcal{C}(\theta)$ is analyzed to ensure the decay property over time. Real parts of $\mathcal{C}(\theta)$ lying on the negative real axis confirm the exponential decay of the solution. A series of numerical experiments were performed to verify the effectiveness of the proposed scheme. These tests include both one-dimensional and two-dimensional cases of advection-free and advection-diffusion flows. They also cover applications to the equal width equation, such as the propagation of a single solitary wave, interactions between two and three solitary waves, undular bore formation, and the Benjamin–Bona–Mahony–Burgers equation. These cases were examined to confirm the theoretical findings and assess the scheme’s accuracy and stability.

We organize the paper as follows: In Section \ref{sec:2}, we provide a compact sixth-order finite-difference approximation to spatial first and second-order derivatives involved in Sobolev-type equations in semi- and fully discrete formulations along with the implementation of Dirichlet boundary conditions. In Section \ref{sec:3} we establish certain conditions to prove the stability of the proposed scheme in the linear case, and hence we prove the solution is bounded in semi- and fully-discrete formulations. In Section \ref{sec:4}, we present numerical results to some examples in one and two-dimensions to validate the theoretical results. Finally, we provide some conclusions in Section \ref{sec:5}.

\section{Compact sixth order scheme (Compact6)}\label{sec:2}
In this section, we aim to develop a numerical scheme with sixth-order accuracy in space and first-order accuracy in time to approximate the solution of Eqn. (\ref{eqn:A}). We define the spatial domain $ \Omega = \{ x \mid a \leq x \leq b \} $, where $ a $ and $ b $ are constants, and discretize it into $ N $ equally spaced intervals so that $ x_j = a + j h $ for $ j = 1, 2, \ldots, N $, where $ h = \dfrac{b-a}{N} $ represents the spatial step size. In the temporal direction, we divide the interval $ [0, T] $ into $ M $ uniform time steps, each of size $ \tau = \dfrac{T}{M} $. The discrete-time levels are then given by $ t^n = n \tau $, where $ n $ is the time index. Let $ u_j^n $ be the approximate solution at $ (x_j, t^n) $. Define $ u^{\prime}_j $ and $ u^{\prime\prime}_j $ as approximations to derivatives $ \dfrac{\partial u}{\partial x} $ and $ \dfrac{\partial^2 u}{\partial x^2} $ at $ x_j $, respectively.
\par
Consider the one-dimensional Sobolev-type equations (\ref{eqn:A}) in the quasi-linear form
\begin{equation}
\label{eqn:Sobolev}
    u_t+f^{\prime}(u) u_x -\gamma u_{xx} - \delta u_{xxt} = \mathrm{g}(x,t).
\end{equation}
To approximate the first and second derivatives, $ u_x $ and $ u_{xx} $, at each grid point, we apply a sixth-order compact finite difference scheme on a linear cell-node grid as presented in \cite{lele1992compact}. For the first derivative, the formulation is as follows:
\begin{equation}
\label{eqn:1a}
\begin{split}
  \dfrac{1}{3} u^{\prime}_{j-1} + u^{\prime}_{j} + \dfrac{1}{3} u^{\prime}_{j+1} &= \dfrac{14}{9} \dfrac{u_{j+1} - u_{j-1}}{2h} + \dfrac{1}{9} \dfrac{u_{j+2} - u_{j-2}}{4h}, \quad 1 \leq j \leq N,\\
  &= \dfrac{1}{h} \left( -\dfrac{1}{36}u_{j-2} -\dfrac{7}{9}u_{j-1}+\dfrac{7}{9}u_{j+1}+\dfrac{1}{36}u_{j+2}\right).
\end{split}
\end{equation}
The local truncation error of the scheme is $\dfrac{-1}{1260} u_j^{(7)}(x) h^6 + \mathcal{O}(h^{8})$.
The Eqn. (\ref{eqn:1a}) can be represented in matrix form as:
\begin{equation}\label{eqn:FD}
\begin{split}
\mathcal{A}_1 \mathbf{u}^{\prime} &= \dfrac{1}{h} \mathcal{B}_1 \mathbf{u}, \\ \mathbf{u}^{\prime} &= \dfrac{1}{h} \mathcal{D}_1 \mathbf{u},
\end{split}
\end{equation}
where $ \mathbf{u}^{\prime} = (u^{\prime}_1, u^{\prime}_2, \dots, u^{\prime}_N)^T $,\,\, $ \mathbf{u} = (u_1, u_2, \dots, u_N)^T $,\,\, $ \mathcal{D}_1 = \mathcal{A}_1^{-1} \mathcal{B}_1 $ is the first derivative differentiation matrix and $ \mathcal{A}_1,\,\, \mathcal{B}_1 \in \mathbb{R}^{N \times N} $ are tri- and penta-diagonal matrices given by
\begin{equation} 
\mathcal{A}_1 = 
\begin{pmatrix}
1 & \dfrac{1}{3} & 0 & \cdots & \cdots & \cdots & 0 \\[0.3cm]
\dfrac{1}{3} & 1 & \dfrac{1}{3} & \cdots & \cdots & \cdots & 0 \\[0.3cm]
0 & \dfrac{1}{3} & 1 & \cdots & \cdots & \cdots & 0 \\[0.1cm]
\vdots & \vdots & \vdots & \ddots & \vdots & \vdots & \vdots \\[0.3cm]
0 & 0 & 0 & \cdots & \cdots & 1 & \dfrac{1}{3} \\[0.3cm]
0 & 0 & 0 & \cdots & \cdots & \dfrac{1}{3} & 1 \\
\end{pmatrix}
, \quad \mathcal{B}_1 =
\begin{pmatrix}
0 & \dfrac{7}{9} & \dfrac{1}{36} & 0 & \cdots & \cdots & 0 \\[0.3cm]
-\dfrac{7}{9} & 0 & \dfrac{7}{9} & \dfrac{1}{36} & \cdots & \cdots & 0 \\[0.3cm]
-\dfrac{1}{36} & -\dfrac{7}{9} & 0 & \dfrac{7}{9} & \dfrac{1}{36} & \cdots & 0 \\[0.3cm]
0 & -\dfrac{1}{36} & -\dfrac{7}{9} & 0 & \dfrac{7}{9} & \cdots & 0 \\[0.3cm]
\vdots & \vdots & \vdots & \vdots & \ddots & \vdots & \vdots \\[0.3cm]
0 & \cdots & 0 & -\dfrac{1}{36} & -\dfrac{7}{9} & 0 & \dfrac{7}{9}\\[0.3cm]
0 & \cdots & 0 & 0 & -\dfrac{1}{36} & -\dfrac{7}{9} & 0 \\
\end{pmatrix}.
\end{equation}
For the second derivative, we utilize the following expression:
\begin{equation}
\label{eqn:2a}
\begin{split}
 \dfrac{2}{11} u^{\prime\prime}_{j-1} + u^{\prime\prime}_{j} + \dfrac{2}{11} u^{\prime\prime}_{j+1} &= \dfrac{12}{11} \dfrac{u_{j+1} - 2u_{j} + u_{j-1}}{h^2} + \dfrac{3}{11} \dfrac{u_{j+2} - 2u_{j} + u_{j-2}}{4h^2}\\ 
 &= \dfrac{1}{h^2} \left(\dfrac{3}{44}u_{j-2} +\dfrac{12}{11}u_{j-1}-\dfrac{51}{22}u_{j}+\dfrac{12}{11}u_{j+1}+\dfrac{3}{44}u_{j+2} \right).
\end{split}
\end{equation}
The local truncation error of the scheme is $\dfrac{-23}{55440} u_j^{(8)}(x) h^6 + \mathcal{O}(h^{8})$. In matrix form, Eqn. (\ref{eqn:2a}) becomes:
\begin{equation}\label{eqn:SD}
\begin{split}
   \mathcal{A}_2 \mathbf{u^{\prime\prime}} &= \dfrac{1}{h^2} \mathcal{B}_2 \mathbf{u}, \\
   \mathbf{u}^{\prime\prime} &= \dfrac{1}{h^2} \mathcal{D}_2 \mathbf{u}, 
\end{split}
\end{equation}
where $ \mathbf{u}^{\prime\prime} = (u^{\prime\prime}_1, u^{\prime\prime}_2, \dots, u^{\prime\prime}_N)^T $,\,\, $ \mathcal{D}_2 = \mathcal{A}_2^{-1} \mathcal{B}_2 $ is the second derivative differentiation matrix and $ \mathcal{A}_2, \mathcal{B}_2 \in \mathbb{R}^{N \times N} $ are tri- and penta-diagonal matrices given by
\begin{equation} 
\mathcal{A}_2 = 
\begin{pmatrix}
1 & \dfrac{2}{11} & 0 & \cdots & \cdots & \cdots & 0 \\[0.3cm]
\dfrac{2}{11} & 1 & \dfrac{2}{11} & \cdots & \cdots & \cdots & 0 \\[0.3cm]
0 & \dfrac{2}{11} & 1 & \cdots & \cdots & \cdots & 0 \\[0.3cm]
\vdots & \vdots & \vdots & \ddots & \vdots & \vdots & \vdots \\[0.3cm]
0 & 0 & 0 & \cdots & \cdots & 1 & \dfrac{2}{11} \\[0.3cm]
0 & 0 & 0 & \cdots & \cdots & \dfrac{2}{11} & 1 \\
\end{pmatrix}
, \, \mathcal{B}_2 =
\begin{pmatrix}
-\dfrac{51}{22} & \dfrac{12}{11} & \dfrac{3}{44} & 0 & \cdots & \cdots & 0 \\[0.3cm]
\dfrac{12}{11} & -\dfrac{51}{22} & \dfrac{12}{11} & \dfrac{3}{44} & \cdots & \cdots & 0 \\[0.3cm]
\dfrac{3}{44} & \dfrac{12}{11} & -\dfrac{51}{22} & \dfrac{12}{11} & \dfrac{3}{44} & \cdots & 0 \\[0.3cm]
0 & \dfrac{3}{44} & \dfrac{12}{11} & -\dfrac{51}{22} & \dfrac{12}{11} & \cdots & 0 \\[0.3cm]
\vdots & \vdots & \vdots & \vdots & \ddots & \vdots & \vdots \\[0.3cm]
0 & \cdots & 0 & \dfrac{3}{44} & \dfrac{12}{11} & -\dfrac{51}{22} & \dfrac{12}{11} \\[0.3cm]
0 & \cdots & 0 & 0 & \dfrac{3}{44} & \dfrac{12}{11} & -\dfrac{51}{22} \\
\end{pmatrix}.
\end{equation}
The derivatives $ u_x $ and $ u_{xx} $ are approximated at the grid points by the vectors $ \mathbf{u^{\prime}} $ and $ \mathbf{u^{\prime\prime}} $, respectively, as defined by equations (\ref{eqn:FD}) and (\ref{eqn:SD}). With these approximations, equation (\ref{eqn:Sobolev}) can be reformulated as:
\begin{equation}
    \mathbf{u}_t + \text{diag}(f^{\prime}(\mathbf{u})) \mathbf{u}^{\prime} - \gamma \, \mathbf{u}^{\prime\prime} - \delta \, \mathbf{u}^{\prime\prime}_t \approx \mathbf{g}, \quad t > 0,
\end{equation}
where $ f^{\prime}(\mathbf{u}) = (f^{\prime}(u_1), f^{\prime}(u_2), \dots, f^{\prime}(u_N))^T $ and $ \mathbf{g} = (\mathrm{g}(x_1, t), \mathrm{g}(x_2, t), \dots, \mathrm{g}(x_N, t))^T $.
\par
By substituting the discrete forms, we arrive at the expression:
\begin{equation}
    \left( I - \dfrac{\delta}{h^2} \mathcal{D}_2 \right) \mathbf{u}_t = -\dfrac{1}{h} \text{diag}(f^{\prime}(\mathbf{u})) \mathcal{D}_1 \mathbf{u} + \dfrac{\gamma}{h^2} \mathcal{D}_2 \mathbf{u} + \mathbf{g}.
\end{equation}
\par
To advance the solution in time, we apply the forward Euler method for the time derivative, resulting in the matrix equation:
\begin{equation}
\label{eqn:Sobolev2}
    \left( I - \dfrac{\delta}{h^2} \mathcal{D}_2 \right) \mathbf{U}^{n+1} = \left( I - \dfrac{\tau}{h} \text{diag}(f^{\prime}(\mathbf{U}^n)) \mathcal{D}_1 + \dfrac{\tau \gamma- \delta}{h^2} \mathcal{D}_2 \right) \mathbf{U}^n + \tau \, \mathbf{g}^n,
\end{equation}
where $ \mathbf{U}^n = (u_1^n, u_2^n, \dots, u_N^n)^T $ represents the solution vector at time level $ t = n \tau $.
\par
At the initial time level, $ n = 0 $, the vector $ \mathbf{U}^0 $ is determined from the initial condition given in equation (\ref{eqn:A_IC}). To obtain $ \mathbf{U}^{n+1} $ at time $ t = (n+1) \tau $, we solve the linear system in equation (\ref{eqn:Sobolev2}). Since $ \mathbf{U}^n $ at time $ t = n \tau $ is already known, the right-hand side of the equation (\ref{eqn:Sobolev2}) can be computed directly, allowing the next time step solution vector $ \mathbf{U}^{n+1} $ to be obtained. Using the Taylor expansion in (\ref{eqn:Sobolev2}), we get the following.

\begin{theorem}\label{eqn:Trun_Err}
Assume that the solution $ u $ of (\ref{eqn:Sobolev}) belongs to the space $ C^8_2((a, b) \times [0, T]) $. Then, as the time step $ \tau \to 0 $ and the spatial step $ h \to 0 $, the Compact6 method (\ref{eqn:Sobolev2})  achieves a local truncation error of order $ \mathcal{O}(\tau + h^6) $.
\end{theorem}
\noindent{\textbf{Boundary conditions:}}\label{sub_sec:BC}
Many computational problems in physics involve non-periodic domains, requiring non-periodic boundary conditions. Therefore, specialized boundary schemes are necessary to accurately compute physical quantities near the edges of the domain. Finite-difference methods are typically used on grids where boundaries align with grid points. The solution is specified at the first and last points, and the governing equations are not solved at these boundaries. Since high-order schemes require neighbouring points that are absent near boundaries, reduced-order or biased schemes are applied at boundary-adjacent points, as they lack sufficient neighbours to support the same numerical approach used in the interior of the grid.
\par
The values for $u_j^{n}$, for $j=1, 2, N-1, N$ are obtained using the boundary condition and hence need to be treated as known quantities. The computation of the spatial derivatives \( u_3' \) and \( u_{N-2}' \) is nontrivial because standard centered finite difference formulas cannot be applied directly near the boundaries. In particular, for a sixth-order accurate centered finite difference approximation of the first derivative, one typically requires symmetric stencils that span three grid points on either side of the evaluation point. However, at the grid points \( x_3 \) and \( x_{N-2} \), such symmetric stencils would require access to values outside the computational domain, which are not available.\\
To address this issue, we use one-sided finite difference formulas that are specially constructed to maintain sixth-order accuracy despite the lack of symmetry in the stencil. These formulas are derived by expanding the solution \( u(x) \) in a Taylor series about the point of interest and determining coefficients that yield a truncation error of \( \mathcal{O}(h^6) \).\\
For example, around the point \( x_3 \), we consider Taylor series expansions of nearby grid values:
As an illustration, to derive a high-order one-sided approximation at the point \( x_3 \), we express the nearby function values \( u_j \), for \( j = 1, 2, \ldots, 6 \), using their Taylor series expansions around that point. These expansions can be compactly written in the form:
\[
u_j = u + jh\, u' + \frac{(jh)^2}{2!} u'' + \frac{(jh)^3}{3!} u^{(3)} + \frac{(jh)^4}{4!} u^{(4)} + \frac{(jh)^5}{5!} u^{(5)} + \mathcal{O}(h^6), \quad j = 1,2,\ldots,6.
\]
The first derivative at \( x_3 \) is then given by:
\[
u_3' = u'(x + 3h) = u' + 3h\, u'' + \frac{9h^2}{2} u^{(3)} + \frac{27h^3}{6} u^{(4)} + \frac{81h^4}{24} u^{(5)} + \mathcal{O}(h^5).
\]
To obtain a one-sided finite difference formula of the form
\[
u_3' \approx \frac{1}{h} \left( a_1 u_1 + a_2 u_2 + a_3 u_3 + a_4 u_4 + a_5 u_5 + a_6 u_6 \right),
\]
we substitute the Taylor expansions into the linear combination and match terms with the Taylor expansion of \( u_3' \). This leads to the following system of equations for the coefficients \( a_1, a_2, \ldots, a_6 \):
\[
\begin{aligned}
& a_1 + a_2 + a_3 + a_4 + a_5 + a_6 = 0, \\
& a_1(1) + a_2(2) + a_3(3) + a_4(4) + a_5(5) + a_6(6) = 1, \\
& a_1 \left( \frac{1}{2} \right) + a_2(2) + a_3 \left( \frac{9}{2} \right) + a_4(8) + a_5 \left( \frac{25}{2} \right) + a_6(18) = 3, \\
& a_1 \left( \frac{1}{6} \right) + a_2 \left( \frac{4}{3} \right) + a_3 \left( \frac{9}{2} \right) + a_4 \left( \frac{64}{6} \right) + a_5 \left( \frac{125}{6} \right) + a_6(36) = \frac{9}{2}, \\
& a_1 \left( \frac{1}{24} \right) + a_2 \left( \frac{1}{3} \right) + a_3 \left( \frac{27}{8} \right) + a_4 \left( \frac{256}{24} \right) + a_5 \left( \frac{625}{24} \right) + a_6 \left( \frac{1296}{24} \right) = \frac{9}{2}, \\
& a_1 \left( \frac{1}{120} \right) + a_2 \left( \frac{2}{15} \right) + a_3 \left( \frac{81}{40} \right) + a_4 \left( \frac{1024}{120} \right) + a_5 \left( \frac{3125}{120} \right) + a_6 \left( \frac{7776}{120} \right) = \frac{81}{24}.
\end{aligned}
\]
After solving the resulting system, we obtain:
\[
a_1 = \frac{1}{20}, \quad a_2 = -\frac{1}{2}, \quad a_3 = -\frac{1}{3}, \quad a_4 = 1, \quad a_5 = -\frac{1}{4}, \quad a_6 = \frac{1}{30}.
\]
Thus, the sixth-order one-sided finite difference approximation of the first derivative at \( x_3 \) is:
\[
u_3' = \frac{1}{h} \left( \frac{1}{20}u_1 - \frac{1}{2}u_2 - \frac{1}{3}u_3 + u_4 - \frac{1}{4}u_5 + \frac{1}{30}u_6 \right) + \mathcal{O}(h^6).
\]
Similarly, the computation of \( u_{N-2}' \) near the right boundary is performed using a sixth-order one-sided stencil that involves grid points \( x_{N-5} \) through \( x_N \). By performing a Taylor series expansion about \( x_{N-2} \) and solving the corresponding coefficient system, we obtain the mirror-image formula:
\[
u_{N-2}' = \frac{1}{h} \left( -\frac{1}{20}u_N + \frac{1}{2}u_{N-1} + \frac{1}{3}u_{N-2} - u_{N-3} + \frac{1}{4}u_{N-4} - \frac{1}{30}u_{N-5} \right) + \mathcal{O}(h^6).
\]
For $\mathbf{u}^{n} = (u_{3}^n, u_{4}^n, \cdots, u_{N-3}^n, u_{N-2}^n)^T$, the first derivative in the matrix form is given by,
\begin{equation}
    \mathbf{A}_1 \mathbf{u^{\prime}} = \dfrac{1}{h} (\mathbf{B}_1 \mathbf{u}^n + \mathbf{C}_1^n),   
\end{equation} 
where, the matrices $\mathbf{A}_1, \mathbf{B}_1 \in \mathbb{R}^{(N-4) \times (N-4)}$ and column vector $\mathbf{C}_1^n \in \mathbb{R}^{(N-4) \times 1}$, are given by 
\begin{equation} 
\mathbf{A}_1 = 
\begin{pmatrix}
1 & 0 & 0 & \cdots & 0 & 0 & 0 & 0 \\[0.3cm]
\dfrac{1}{3} & 1 & \dfrac{1}{3} & \ddots & & & & 0 \\[0.3cm]
0 & \dfrac{1}{3} & 1 & \dfrac{1}{3} & \ddots & & & 0 \\[0.3cm]
\vdots & \ddots & \ddots & \ddots & \ddots & \ddots & & \vdots \\[0.3cm]
0 & & \ddots & \dfrac{1}{3} & 1 & \dfrac{1}{3} & 0 & 0 \\[0.3cm]
0 & & & \ddots & \dfrac{1}{3} & 1 & \dfrac{1}{3} & 0 \\[0.3cm]
0 & & & & \ddots & \dfrac{1}{3} & 1 & \dfrac{1}{3} \\[0.3cm]
0 & 0 & 0 & \cdots & 0 & 0 & 0 & 1 \\
\end{pmatrix}
, \,
\mathbf{B}_1 = 
\begin{pmatrix}
-\dfrac{1}{3} & 1 & -\dfrac{1}{4} & \dfrac{1}{30} & \cdots & 0 & 0 & 0 \\[0.3cm]
-\dfrac{7}{9} & 0 & \dfrac{7}{9} & \dfrac{1}{36} & \ddots & & & 0 \\[0.3cm]
-\dfrac{1}{36} & -\dfrac{7}{9} & 0 & \dfrac{7}{9} & \dfrac{1}{36} & \ddots & & 0 \\[0.3cm]
0 & -\dfrac{1}{36} & -\dfrac{7}{9} & 0 & \dfrac{7}{9} & \dfrac{1}{36} & \ddots & 0 \\[0.3cm]
\vdots & \ddots & \ddots & \ddots & \ddots & \ddots & \ddots & \vdots \\[0.3cm]
0 & & \ddots & -1 & -\dfrac{7}{9} & 0 & \dfrac{7}{9} & \dfrac{1}{36} \\[0.3cm]
0 & & & 0 & -\dfrac{1}{36} & -\dfrac{7}{9} & 0 & \dfrac{7}{9} \\[0.3cm]
0 & 0 & 0 & \cdots & -\dfrac{1}{30} & \dfrac{1}{4} & -\dfrac{1}{36} & \dfrac{1}{3} \\
\end{pmatrix},
\end{equation}
\begin{equation} 
\mathbf{C}_1^n = \Biggl( \dfrac{1}{20}u_1^n-\dfrac{1}{2}u_2^n, -\dfrac{1}{36}u_2^n,\, 0, \cdots, 0,\, \dfrac{1}{36}u_{N-1}^n, -\dfrac{1}{20}u_{N}^n +\dfrac{1}{2}u_{N-1}^n \Biggr)^T.
\end{equation}
Similarly, the one-sided finite difference approximation of sixth-order for the second derivatives at the grid points $x_3$, and $x_{N-2}$ are given by
\begin{equation}
\begin{split} 
u^{\prime\prime}_{3} &= \dfrac{1}{h^2} \biggl[ -\dfrac{1}{12}u_{1} +\dfrac{4}{3}u_{2}-\dfrac{5}{2}u_{3}+\dfrac{4}{3}u_{4}-\dfrac{1}{12}u_{5} \biggr],\\
u^{\prime\prime}_{N-2} &= \dfrac{1}{h^2} \biggl[ -\dfrac{1}{12}u_{N} +\dfrac{4}{3}u_{N-1}-\dfrac{5}{2}u_{N-2} +\dfrac{4}{3}u_{N-3} -\dfrac{1}{12}u_{N-4} \biggr].
\end{split}
\end{equation}
The second derivative in the matrix form is given by,
\begin{equation}
    \mathbf{A}_2 \mathbf{u^{\prime\prime}} = \dfrac{1}{h} (\mathbf{B}_2 \mathbf{u}^n +\mathbf{C}_2^{n}),   
\end{equation}
where, the matrix $\mathbf{A}_2$, $\mathbf{B}_2$ and column vector $\mathbf{C}_2^{n}$, are given by
\begin{equation} 
\mathbf{A}_2 = 
\begin{pmatrix}
1 & 0 & 0 & \cdots & 0 & 0 & 0 & 0 \\[0.3cm]
\dfrac{2}{11} & 1 & \dfrac{2}{11} & \ddots & & & & 0 \\[0.3cm]
0 & \dfrac{2}{11} & 1 & \dfrac{2}{11} & \ddots & & & 0 \\[0.3cm]
\vdots & \ddots & \ddots & \ddots & \ddots & \ddots & & \vdots \\[0.3cm]
0 & & \ddots & \dfrac{2}{11} & 1 & \dfrac{2}{11} & 0 & 0 \\[0.3cm]
0 & & & \ddots & \dfrac{2}{11} & 1 & \dfrac{2}{11} & 0 \\[0.3cm]
0 & & & & \ddots & \dfrac{2}{11} & 1 & \dfrac{2}{11} \\[0.3cm]
0 & 0 & 0 & \cdots & 0 & 0 & 0 & 1 \\
\end{pmatrix}
,\,
\mathbf{B}_2 = 
\begin{pmatrix}
-\dfrac{5}{2} & \dfrac{4}{3} & -\dfrac{1}{12} & 0 & \cdots & 0 & 0 & 0 \\[0.3cm]
\dfrac{12}{11} & -\dfrac{51}{22} & \dfrac{12}{11} & \dfrac{3}{44} & \ddots & & & 0 \\[0.3cm]
\dfrac{3}{44} & \dfrac{12}{11} & -\dfrac{51}{22} & \dfrac{12}{11} & \dfrac{3}{44} & \ddots & & 0 \\[0.3cm]
0 & \dfrac{3}{44} & \dfrac{12}{11} & -\dfrac{51}{22} & \dfrac{12}{11} & \dfrac{3}{44} & \ddots & 0 \\[0.3cm]
\vdots & \ddots & \ddots & \ddots & \ddots & \ddots & \ddots & \vdots \\[0.3cm]
0 & & \ddots & \dfrac{3}{44} & \dfrac{12}{11} & -\dfrac{51}{22} & \dfrac{12}{11} & \dfrac{3}{44} \\[0.3cm]
0 & & & 0 & \dfrac{3}{44} & \dfrac{12}{11} & -\dfrac{51}{22} & \dfrac{12}{11} \\[0.3cm]
0 & 0 & 0 & \cdots & 0 & -\dfrac{1}{12} & \dfrac{4}{3} & -\dfrac{5}{2} \\
\end{pmatrix},
\end{equation}

\begin{equation} 
\mathbf{C}_2^{n} = \Biggl( -\dfrac{1}{12}u_1^n+\dfrac{4}{3}u_2^n,\, \dfrac{3}{44}u_2^n,\, 0, \cdots, 0,\, \dfrac{3}{44}u_{N-1}^n ,\, -\dfrac{1}{12}u_{N}^n +\dfrac{4}{3}u_{N-1}^n \Biggr)^T.
\end{equation}
Thus, the method (\ref{eqn:Sobolev2}) now takes the form
\begin{equation}
\begin{split}
       \Biggl( I -\dfrac{\delta}{h^2} \mathbf{A}_2^{-1}\mathbf{B}_2 \Biggr) \mathbf{u}_t &=  -\dfrac{1}{h} \text{diag}(f^{\prime}(\mathbf{u}^n))\mathbf{A}_1^{-1}\mathbf{B}_1 \mathbf{u}^n + \dfrac{\gamma}{h^2} \mathbf{A}_2^{-1}\mathbf{B}_2  \mathbf{u}^n + \mathbf{g}^n\\
        & + \dfrac{\delta}{h^2} \mathbf{A}_2^{-1}\mathbf{C}_2^{n+1} -\dfrac{1}{h}\text{diag}(f^{\prime}(\mathbf{u}^n))\mathbf{A}_1^{-1}\mathbf{C}_1^n + \dfrac{\gamma}{h^2} \mathbf{A}_2^{-1}\mathbf{C}_2^{n},\\  
       \Biggl( h^2I -\delta \mathbf{A}_2^{-1}\mathbf{B}_2\Biggr) \mathbf{u}_t &= -h \, \text{diag}(f^{\prime}(\mathbf{u}^n))\mathbf{A}_1^{-1}\mathbf{B}_1 \mathbf{u}^n + \gamma \, \mathbf{A}_2^{-1}\mathbf{B}_2  \mathbf{u}^n + h^2\mathbf{g}^n\\
        & + \delta \, \mathbf{A}_2^{-1}\mathbf{C}_2^{n+1} -h\, \text{diag}(f^{\prime}(\mathbf{u}^n))\mathbf{A}_1^{-1}\mathbf{C}_1^n + \gamma \, \mathbf{A}_2^{-1}\mathbf{C}_2^{n}.\\ 
\end{split}       
\end{equation}
For the time derivative, we apply the forward Euler method, leading to the matrix equation
\begin{equation}
\begin{split}
  \Biggl( I - \dfrac{\delta}{h^2} \mathbf{A}_2^{-1}\mathbf{B}_2\Biggr) \mathbf{U}^{n+1} &= \Biggl(I - \dfrac{\tau}{h} \text{diag}(f^{\prime}(\mathbf{U}^n))\mathbf{A}_1^{-1}\mathbf{B}_1 + \dfrac{1}{h^2}\biggl(\tau \gamma - \delta  \biggr) \mathbf{A}_2^{-1}\mathbf{B}_2 \Biggr) \mathbf{U}^{n} + \tau \mathbf{g}^n  \\
  & + \dfrac{\delta}{h^2} \mathbf{A}_2^{-1}\mathbf{C}_2^{n+1} -\dfrac{\tau}{h}\text{diag}(f^{\prime}(\mathbf{u}^n))\mathbf{A}_1^{-1}\mathbf{C}_1^n + \dfrac{1}{h^2}\biggl(\tau \gamma - \delta  \biggr) \mathbf{A}_2^{-1}\mathbf{C}_2^{n},
  \end{split} 
\end{equation}
where the approximate solution vector takes the form $\mathbf{U}^{n} = (u_{3}^n, u_{4}^n, \cdots, u_{N-3}^n, u_{N-2}^n)^T$. The extension of the present scheme to the two-dimensional case is carried out in a dimension-by-dimension manner for spatial discretization.

\section{Error analysis}\label{sec:3}
In this section, we perform a stability analysis of the Compact6 method, which is detailed out in Section \ref{sec:2}. This is achieved through von-Neumann analysis, focusing specifically on a Sobolev-type equation governed by equation (\ref{eqn:A}), where we assume $f(u)$ is linear and $\mathrm{g}$ is zero. The primary equation for this analysis is:
\begin{equation}\label{eqn:linearSobolev}
    u_t+\alpha u_x = \gamma u_{xx} + \delta u_{xxt}, \quad (x,t) \in (0,b) \times (0, T],
\end{equation}
where $\alpha$ is a given real number. Given the linearity of the problem, it is sufficient to consider a single Fourier mode for simplification:
\begin{equation}\label{eqn:Fourier1}
    u_j(t) = \hat{u}(t) e^{ij\theta}, 
\end{equation}
where $\theta = \omega h$, $\omega$ denotes the wave number and $i = \sqrt{-1} $. By substituting this Fourier mode into the Compact6 method's approximations for the first and second derivatives, we replace $u_x$ and $u_{xx}$ in equation (\ref{eqn:linearSobolev}), yielding an ordinary differential equation:
\begin{equation}\label{eqn:AF}
    \dfrac{d u_j}{dt} = \mathcal{C}(\theta) u_j.
\end{equation}
which is the semi-discrete form of (\ref{eqn:linearSobolev}). Here, $\mathcal{C}(\theta)$ is identified as the amplification factor. 
\par Substituting (\ref{eqn:Fourier1}) into the expression (\ref{eqn:1a}) of the first derivative of the Compact6 method, we get
\begin{equation}\label{eqn:Compact6-1}
     \begin{aligned}
              \mathbf{u}^{\prime}_j &= \dfrac{i}{9h} \biggl[\dfrac{14 \sin(\theta) + \dfrac{1}{2} \sin(2\theta)}{1+\dfrac{2}{3} \cos(\theta)} \biggr] u_j(t),\\
               &= \dfrac{i}{6h} \biggl[\dfrac{28 \sin(\theta) + \sin(2\theta)}{3+2 \cos(\theta)} \biggr] u_j(t), \quad j = 3,4,\cdots, N-2, \\
        \end{aligned}
\end{equation}
similarly, substituting (\ref{eqn:Fourier1}) into the expression (\ref{eqn:2a}) of the second derivative of the Compact6 method, we get
\begin{equation}
     \begin{aligned}
              \mathbf{u}^{\prime\prime}_j &= \dfrac{1}{11h^2} \biggl[\dfrac{24 \cos(\theta) + \dfrac{3}{2} \cos(2\theta) - \dfrac{51}{2}}{1+\dfrac{4}{11} \cos(\theta)} \biggr] u_j(t),\\
               &= \dfrac{1}{2h^2} \biggl[\dfrac{48 \cos(\theta) + 3 \cos(2\theta) - 51}{11+4 \cos(\theta)} \biggr] u_j(t), \quad j = 3,4,\cdots, N-2. \\
        \end{aligned}
\end{equation}
Upon substituting these expressions in (\ref{eqn:linearSobolev}), we get
\begin{equation}
    (u_j)_t-\delta \bigl( \mathbf{u}_j^{\prime\prime}\bigr)_t = -\alpha  \mathbf{u}_j^{\prime} + \gamma \mathbf{u}_j^{\prime\prime}, 
\end{equation}
and we arrive at the semi-discrete form (\ref{eqn:AF}), where the amplification factor takes the form
\begin{equation}
    \mathcal{C}(\theta)=\displaystyle \dfrac{\dfrac{\gamma}{2h^2} \biggl(\dfrac{48 \cos(\theta) + 3 \cos(2\theta) - 51}{11+4 \cos(\theta)} \biggr) - \dfrac{i \alpha}{6h} \biggl(\dfrac{28 \sin(\theta) + \sin(2\theta)}{3+2 \cos(\theta)} \biggr)}{1-\dfrac{\delta}{2h^2} \biggl(\dfrac{48 \cos(\theta) + 3 \cos(2\theta) - 51}{11+4 \cos(\theta)} \biggr)}.
\end{equation}

\begin{theorem}\label{Thm:Decay_estimate}
Suppose the initial data $u_0$ is bounded. Then, for each $j = 1, \dots, N$, the solution $u_j$ of the semi-discrete system (\ref{eqn:AF}) satisfies the following inequality:
\begin{equation}\label{Decay_estimate}
    \Vert u(t) \rVert_{\infty} \leq e^{\mathfrak{Re}(\mathcal{C} (\theta))t} \, \Vert u_0 \rVert_{\infty}, \quad t \in [0, T].
\end{equation}
\end{theorem}
\begin{proof}
Integrating the equation (\ref{eqn:AF}) over the time interval $[0, t]$. At $t=0$, let $u_j(0) = u_0(x_j)$, we have
\begin{equation*}
    \int_{u_0(x_j)}^{u_j(t)} \dfrac{d u_j}{u_j} =\int_{0}^{t} \mathcal{C}(\theta) dt.
\end{equation*}
On simplification, we get
\begin{equation*}
    u_j(t) = e^{\mathcal{C}(\theta)t} \, u_0(x_j).
\end{equation*}
By taking the modulus on both sides, we obtain
\begin{equation*}
    |u_j(t)| \leq e^{\mathfrak{Re}(\mathcal{C}(\theta))t} \, \Vert u_0 \rVert_{\infty}.
\end{equation*}
This inequality holds for every $j = 1, \dots, N$, thus establishing the desired result. 
\end{proof}
\par
Upon applying the forward Euler scheme to equation (\ref{eqn:AF}), we derive the expression for the fully discrete scheme given by
\begin{equation}
    u_j^{n+1} = \mathcal{L}(\theta) u_j^n,
\end{equation}
where $\mathcal{L}(\theta)$ is given by
\begin{equation}\label{eqn:K_1}
    \mathcal{L}(\theta) = \dfrac{\biggl(1+\dfrac{\gamma\tau - \delta}{2h^2} \bigl(\dfrac{48 \cos(\theta) + 3 \cos(2\theta) - 51}{11+4 \cos(\theta)} \bigr)\biggr) - \dfrac{i \alpha \tau}{6h} \bigl(\dfrac{28 \sin(\theta) + \sin(2\theta)}{3+2 \cos(\theta)} \bigr)}{1-\dfrac{\delta}{2h^2} \bigl(\dfrac{48 \cos(\theta) + 3 \cos(2\theta) - 51}{11+4 \cos(\theta)} \bigr)}.
\end{equation}
In the context of von Neumann stability analysis, the amplification factor is generally expected to stay within a limit of 1. Nevertheless, this condition can be relaxed to $1 + C \tau$, where $C$ is a positive constant \cite{leveque2002finite, strikwerda2004finite}. The following theorem presents $L_2$-stability estimates for the method proposed, specifically when applied to the linear equation (\ref{eqn:linearSobolev}).
\begin{theorem}\label{Thm:dt_estimate}
The Compact6 method applied to equation (\ref{eqn:linearSobolev}) with the forward Euler method for time discretization remains stable if there exists a constant $C > 0$ such that
\begin{equation}\label{thm:dt1}
    \tau \leq \dfrac{2\biggl( 1 - \dfrac{\delta P}{2 h^2}\biggr)\biggl( C - \dfrac{C\delta P}{2 h^2}- \dfrac{\gamma P}{2 h^2} \biggr)}{\biggl(\dfrac{ \gamma^2 P^2}{4 h^4} + \dfrac{\alpha^2 Q^2}{36 h^2}\biggr)},
\end{equation}
where the constants $P$ and $Q$ are defined by 
 \begin{equation}\label{eqn:AB1}
     P = \biggl(\dfrac{48 \cos(\theta) + 3 \cos(2\theta) - 51}{11+4 \cos(\theta)} \biggr), \quad Q = \biggl(\dfrac{28 \sin(\theta) + \sin(2\theta)}{3+2 \cos(\theta)} \biggr).
 \end{equation}
 \begin{proof}
     The stability of the Compact6 method holds if 
     \begin{equation}
         |\mathcal{L}(\theta)|^2 \leq (1 + C\, \tau)^2 \cong  1 + 2 C\, \tau,
     \end{equation}
where $C = \mathfrak{Re} \dfrac{\biggl(\dfrac{\gamma P}{2h^2} - \dfrac{i\alpha Q}{6h} \biggr)}{\biggl(1 - \dfrac{\delta P}{2 h^2}\biggr)}$. Using equation (\ref{eqn:K_1}) and the notation defined in (\ref{eqn:AB1}), we obtain
     \begin{equation}
         \dfrac{\gamma^2 \tau^2 P^2}{4 h^4} + 2 \biggl(1 - \dfrac{\delta P}{2 h^2}\biggr)\dfrac{\gamma \tau P}{2 h^2} + \dfrac{\alpha^2 \tau^2 Q^2}{36 h^2} \leq 2C\, \tau \biggl(1 - \dfrac{\delta P}{2 h^2}\biggr)^2.
     \end{equation}
     Simplifying the above estimate results in (\ref{thm:dt1}).
 \end{proof}
\end{theorem}

\begin{remark}\label{remark:1}
    By setting $\gamma = \delta =0$ in equation (\ref{eqn:K_1}), we derive the expression 
\begin{equation}
    |\mathcal{L}(\theta)|^2 = 1 + \dfrac{\alpha^2 \tau^2}{36 h^2} 
    \biggl(\dfrac{\displaystyle 28 \sin(\theta) + \sin(2\theta)}{\displaystyle 3+2 \cos(\theta)} \biggr)^2 > 1, 
    \quad \theta \neq l \pi, \quad l \in \mathbb{Z}^+.
\end{equation}
   This indicates that the Compact6 method is unconditionally unstable when $C = 0$. 
\end{remark}
\subsection{Nonlinear Stability Analysis}
The von Neumann method is a common tool for analyzing $L^2$-stability, but it is primarily applicable to linear schemes. Its application to nonlinear problems is generally limited. For nonlinear equations, stability can sometimes be analyzed using energy estimates, or in specific scenarios, one may consider $L^\infty$-stability instead. For instance, $L^\infty$-stability has been established for the Burgers–Fisher equation using cubic B-spline quasi-interpolation methods \cite{zhu2010numerical}. However, in the context of Sobolev-type equations, the presence of the dispersive term $u_{xxt}$ makes the stability analysis challenging using these standard approaches \cite{kumar2016b}.\par
An approach to investigate the $L^2$-stability of Sobolev-type equations with $\mathrm{g} = 0$ is through a linearization process \cite{strang1964accurate}. For example, consider the equation:
\begin{equation}
    u_t + f(u)_x - \gamma u_{xx} - \delta u_{xxt} = 0.
\end{equation}
Assume a smooth solution $u$ exists for some small interval $[t, t+\tau]$ within $(0, T]$. To approximate $u(x, t + \tau)$, we expand it using a Taylor series and ignore higher-order terms, yielding:
\begin{equation}\label{eqn:NLS1}
    u(x, t+\tau) \approx u(x, t) + \tau \, v(x, t),
\end{equation}
where $v = u_t$.  Expand $f(u)$ using Taylor series around the known solution $u(x,t)$ as
\begin{equation}
    f(u)\bigg|_{(x, t+\tau)} \approx f(u) + f^\prime(u) \Delta u + \dfrac{1}{2} f^{\prime \prime}(u)(\Delta u)^2 + \cdots \bigg|_{(x, t)} ,
\end{equation}
where $\Delta u \approx \tau \, v(x, t)$. For simplicity, higher-order terms ($(\Delta u)^2, (\Delta u)^3, \cdots $) are neglected, as they are small. Thus, 
\begin{equation}\label{eqn:NLS2}
    f(u)\bigg|_{(x, t+\tau)} \approx f(u) + f^\prime(u) (\tau\, v(x, t)).
\end{equation}
Differentiate $f(u)$ with respect to $x$ and using the Taylor expansion we get
\begin{equation}\label{eqn:NLS3}
\begin{split}
    f(u)_x\bigg|_{(x, t+\tau)} & \approx \dfrac{\partial}{\partial x} \biggl(f(u) + f^\prime(u) (\tau\, v(x, t))\biggr)\\    
    &= f^\prime(u)\, u_x + \tau \biggl(f^\prime(u) v_x + f^{\prime \prime}(u)\, u_x\, v(x,t)\biggr).
\end{split}
\end{equation}
Substituting Eqn. (\ref{eqn:NLS1}) and (\ref{eqn:NLS3}) into the governing equation  and neglecting higher-order terms in $\tau$, we get
\begin{equation}
v_t + f^{\prime}(u) v_x - \gamma v_{xx} - \delta v_{xxt} + f^{\prime \prime}\,(u)\,u_x\,v(x,t) = 0.
\end{equation}
By assuming that the solution $u$ at time level $t$ is known, the above equation can be analyzed for $L^2$-stability using the von Neumann approach, as described in Section \ref{sec:3}.

\section{Numerical results}\label{sec:4}
In this section, the performance of the sixth-order compact scheme is evaluated using several 1D and 2D numerical examples. Forward Euler is employed for time integration, with the time step chosen as $\tau = h^{6}$. Additionally, we compare the performance of our scheme with the Cubic B-spline quasi-interpolation (CBSQI) and the Improved Cubic B-spline quasi-interpolation (ICBSQI) methods proposed in \cite{kumar2016b}, where a fixed time step of $\tau = 10^{-4}$ is used.
\begin{example}
\normalfont
(Advection Free Flow) \\
Consider the 1D linear Sobolev-type equation
\begin{equation}\label{example:1}
    u_t = u_{xx} + u_{xxt}, \quad (x,t) \in [0,30] \times [0, T],
\end{equation}
with the initial condition $u(x,0) = \sin(x)$, for $x \in [0, 30]$. The exact solution for this equation is given by $u(x,t) = e^{-t/2}\sin(x)$ over the domain $(x,t) \in [0,30] \times [0, T]$. For this initial value problem, we apply the Compact6 method, implementing boundary corrections as described in Section \ref{sub_sec:BC} using the exact solution for boundary terms. The numerical solution is computed at $T = 1$ over the spatial domain $[0, 30]$. In Table~\ref{Table:E1a}, the $L^{\infty}$, $L^{1}$, and $L^{2}$ errors and the convergence rates of the method are presented.  As shown, the Compact6 method demonstrates numerical convergence to the exact solution at a rate of six, aligning well with theoretical results from Theorem \ref{eqn:Trun_Err}. The numerical solutions and absolute error for $N = 150$ at time $T = 1$ is depicted in Fig.~\ref{Figure:E1a}. Table~\ref{Table:E1b} presents the error analysis and order of convergence at the final time \( T = 1 \), using a fixed spatial resolution of \( N = 300 \). The table reports numerical errors in the \( L^{\infty} \), \( L^1 \), and \( L^2 \) norms for decreasing time step sizes \( \tau = 10^{-1}, 10^{-2}, 10^{-3}, 10^{-4}, 10^{-5} \). As shown in the table, the numerical scheme exhibits approximately first-order convergence in time across all norms, confirming the expected temporal accuracy of the method. Furthermore, Fig.~\ref{Figure:E1d} compares the $L^{\infty}$ errors of the CBSQI, ICBSQI, and Compact6 schemes in $T = 1$ with $\tau = 0.0001$, indicating that the Compact6 scheme yields significantly smaller errors than the other two schemes.
\par

A numerical verification of the exponential decay in the solution, as per Theorem \ref{Thm:Decay_estimate}, is conducted over the domain \( [0, 2\pi] \) divided into 50 grid points. The graph in Fig.~\ref{Figure:E1b}(i) illustrates the decay of the solution over time, and Fig.~\ref{Figure:E1b}(ii) compares the left and right sides of estimate (\ref{Decay_estimate}). This comparison shows that the numerical solution’s magnitude remains within the bounds given by the right-hand side of the estimate, confirming exponential decay with respect to time. 
\par
Next, we examine the time-step constraints for stability in the Compact6 method. The stability limit based on (\ref{thm:dt1}) becomes $\tau \leq 2 \biggl(1 + \dfrac{7 h^2}{48} \biggr)$. To validate this numerically, we run an experiment over the domain $[0, \pi]$ discretized into 100 intervals. In this setup, the stability bound allows $ \tau = 2.00115 $. The solutions obtained using the proposed method at $ T = 1000 $ are illustrated in Fig.~\ref{Figure:E1c}. Due to the exponential time decay of the solution, the exact solution approaches zero at this large time. Fig.~\ref{Figure:E1c}(a) shows the solutions with $\tau = 2$, where stability, as predicted by the stability estimates, is confirmed. Although the numerical solution deviates from the exact solution, it is bounded within $10^{-4}$. By extending the time interval up to $T = 10,000$, Fig.~\ref{Figure:E1c}(c) shows that the solutions remain bounded within $10^{-4}$. In contrast, taking $ \tau = 2.1 $, just slightly beyond the stability bound, leads to solution blow-up, as seen in Figs.~\ref{Figure:E1c}(b) and \ref{Figure:E1c}(d), where the maximum solution value increases to $10^{16}$ and reaches $10^{193}$ by $T = 10,000$.
\begin{table}[htbp!]
\centering
\captionof{table}{Errors and order of convergence for Example \ref{example:1}.}
\setlength{\tabcolsep}{0pt}
\begin{tabular*}{\textwidth}{@{\extracolsep{\fill}} l *{8}{c} }
\toprule
 \textbf{N} & $\boldsymbol{L^{\infty}}$\textbf{-error} & \textbf{Rate} & $\boldsymbol{L^{1}}$\textbf{-error} & \textbf{Rate} & $\boldsymbol{L^{2}}$\textbf{-error} & \textbf{Rate}\\
\midrule
 40& 1.9599e-02 & - & 1.0490e-02 & - & 1.0883e-02 & - \\
80& 2.7099e-04 & 6.1764 & 1.4551e-04 & 6.1716& 1.5698e-04 & 6.1154 \\
160& 3.9937e-06 & 6.0844 & 2.1853e-06 & 6.0572& 2.3905e-06 & 6.0371 \\
320& 6.0537e-08 & 6.0438 & 3.3416e-08 & 6.0312& 3.6895e-08 & 6.0177 \\
\bottomrule
\end{tabular*}\label{Table:E1a}
\end{table}
\begin{table}[htbp!]
\centering
\captionof{table}{Errors and order of convergence for Example \ref{example:1} at $T=1$ and $N=300$.}
\setlength{\tabcolsep}{0pt}
\begin{tabular*}{\textwidth}{@{\extracolsep{\fill}} l *{8}{c} }
\toprule
 \textbf{$\boldsymbol{\tau}$} & $\boldsymbol{L^{\infty}}$\textbf{-error} & \textbf{Rate} & $\boldsymbol{L^{1}}$\textbf{-error} & \textbf{Rate} & $\boldsymbol{L^{2}}$\textbf{-error} & \textbf{Rate}\\
\midrule
 1e-01& 8.9386e-03 & - & 4.9551e-03 & - & 5.4720e-03 & - \\
1e-02& 8.7621e-04 & 1.0087 & 4.8405e-04 & 1.0102& 5.3438e-04 & 1.0103 \\
1e-03& 8.7450e-05 & 1.0009 & 4.8294e-05 & 1.0010& 5.3314e-05 & 1.0010 \\
1e-04& 8.7434e-06 & 1.0001 & 4.8283e-06 & 1.0001& 5.3302e-06 & 1.0001 \\
1e-05& 8.7443e-07 & 1.0000 & 4.8285e-07 & 1.0000& 5.3302e-07 & 1.0000 \\
\bottomrule
\end{tabular*}
\label{Table:E1b}
\end{table}
\begin{figure}[htbp!]
    \centering
    \begin{minipage}[b]{0.45\linewidth}
      \includegraphics[width=\linewidth]{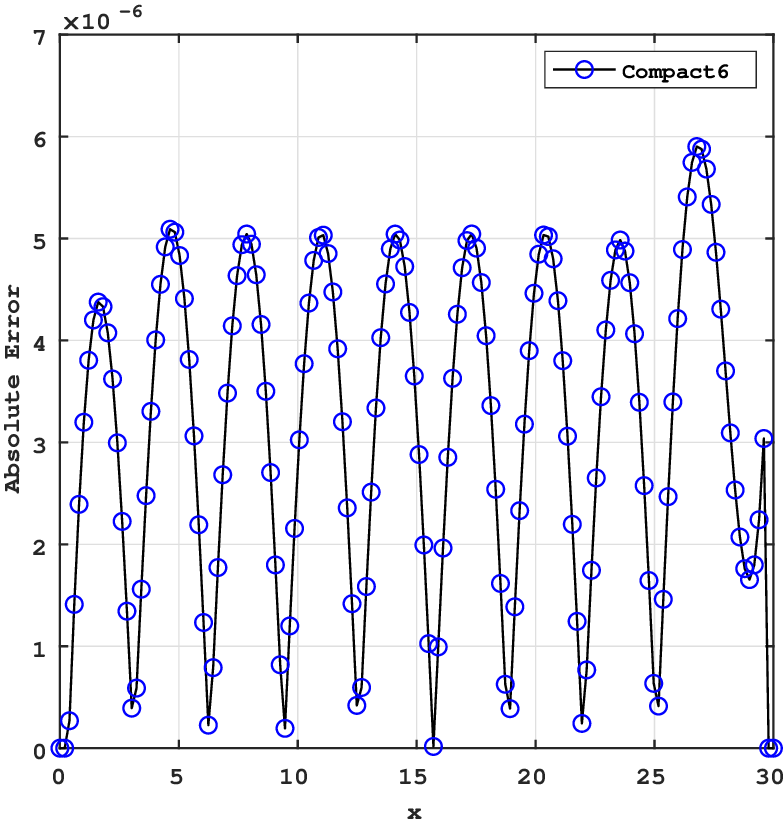}
      \subcaption*{(i) Absolute Error Distribution}
    \end{minipage}\hfill
    \begin{minipage}[b]{0.45\linewidth}
      \includegraphics[width=\linewidth]{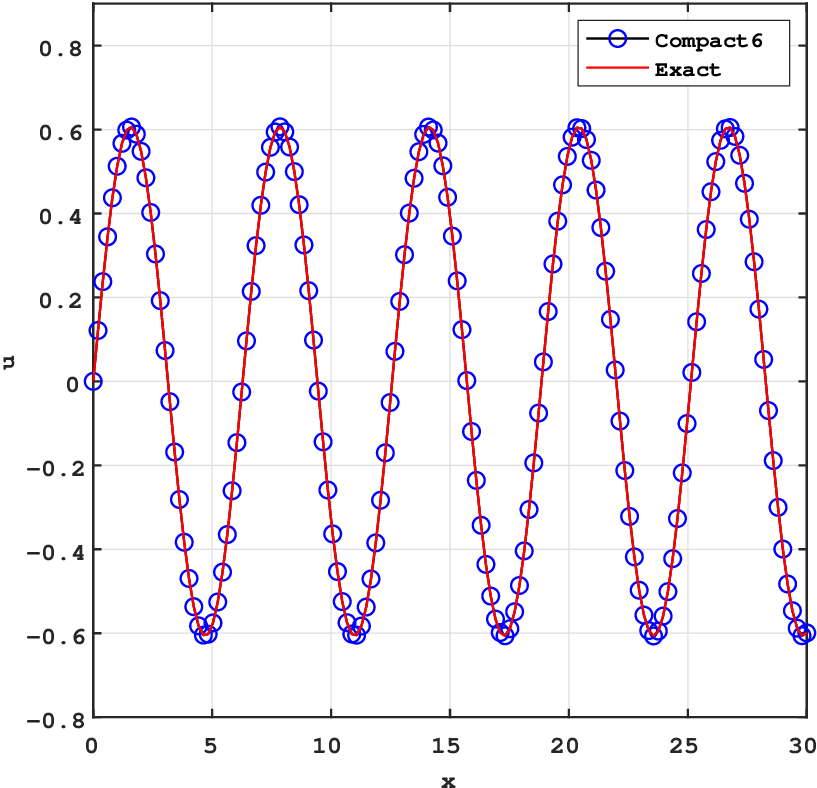}
      \subcaption*{(ii) Numerical vs. Exact Solution}
    \end{minipage}
    \caption{(i) Absolute error distribution for the Compact6 method in Example \ref{example:1} at $T = 1$. (ii) Comparison of the numerical solution from the Compact6 method (o symbols) with the exact solution (solid line) at $T = 1$.}
    \label{Figure:E1a}
\end{figure}  
\begin{figure}[htbp!]  
    \centering
    \begin{minipage}[b]{0.45\linewidth}
      \includegraphics[width=\linewidth]{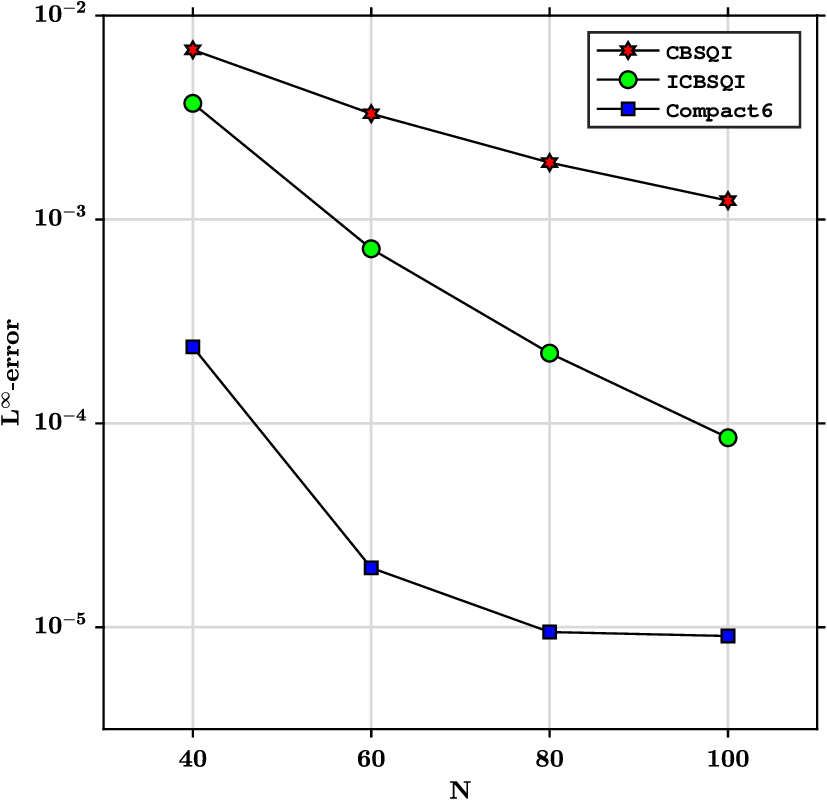}
    \end{minipage}\hfill
    \caption{ Comparison of CBSQI, ICBSQI and Compact6 schemes in terms of $L^{\infty}$ errors (in $log10$ scale) for Example~\ref{example:1} at $T=1$ and $\tau =0.0001$.}
    \label{Figure:E1d} 
\end{figure}
\begin{figure}[htbp!]  
    \centering
    \begin{minipage}[b]{0.45\linewidth}
      \includegraphics[width=\linewidth]{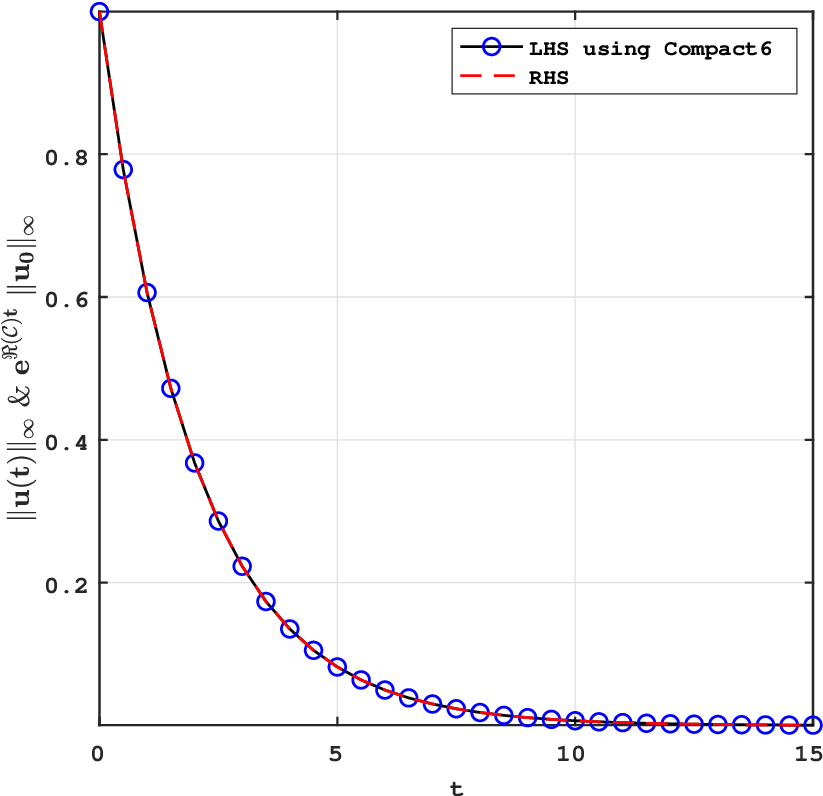}
      \subcaption*{(i) Exponential Decay Validation}
    \end{minipage}\hfill
    \begin{minipage}[b]{0.45\linewidth}
      \includegraphics[width=\linewidth]{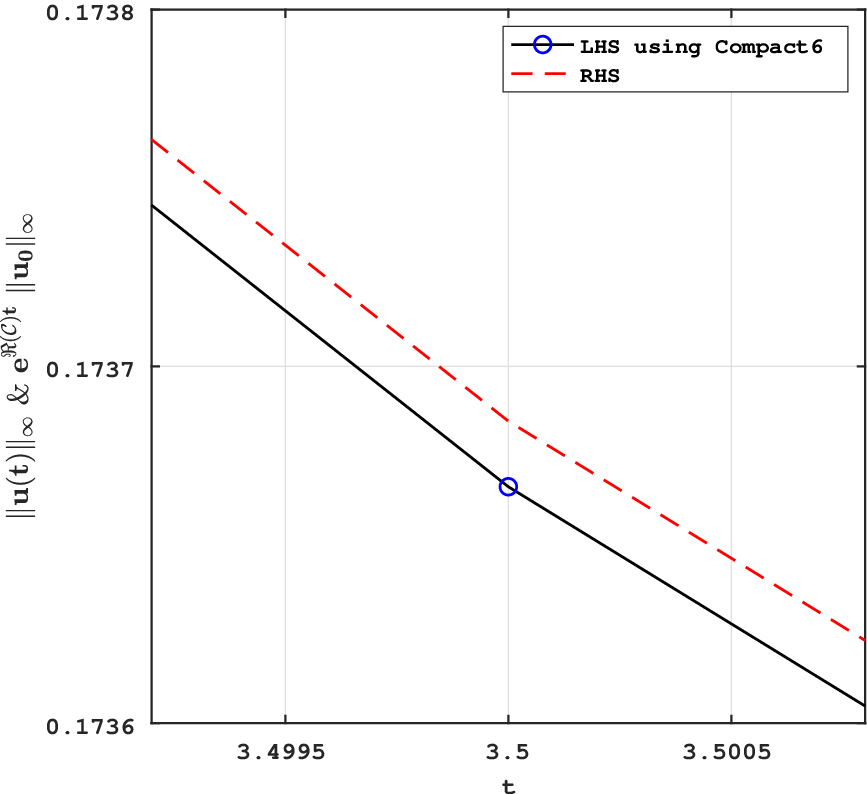}
      \subcaption*{(ii) Decay Estimate Close-Up}
    \end{minipage}
    \caption{(i) Numerical validation of the exponential decay from Theorem \ref{Thm:Decay_estimate} using the Compact6 method. (ii) Zoom around $t=3.5$ showing the decay estimate (\ref{Decay_estimate}). -o- line: left-hand side of (\ref{Decay_estimate}), dashed line: right-hand side.}
    \label{Figure:E1b}
\end{figure}  
\begin{figure}[htbp!]  
    \centering
    \begin{minipage}[b]{0.45\linewidth}
      \includegraphics[width=\linewidth]{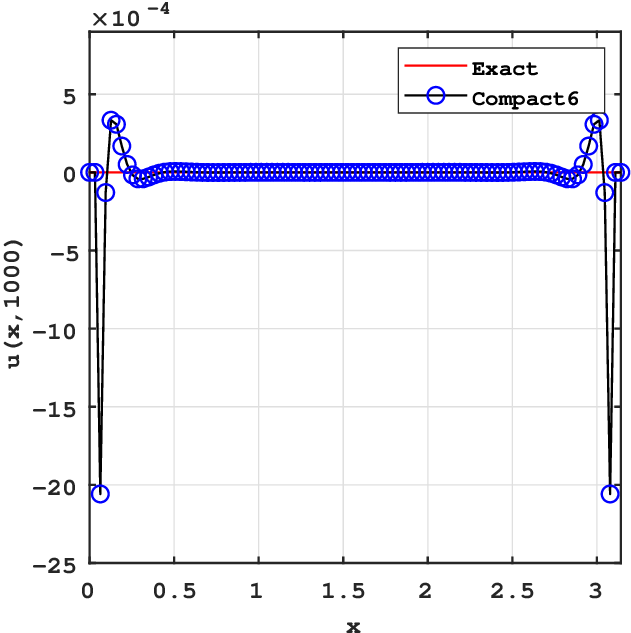}
      \subcaption{Stable Solution, $\tau = 2$, $T = 1000$}
    \end{minipage}\hfill
    \begin{minipage}[b]{0.45\linewidth}
      \includegraphics[width=\linewidth]{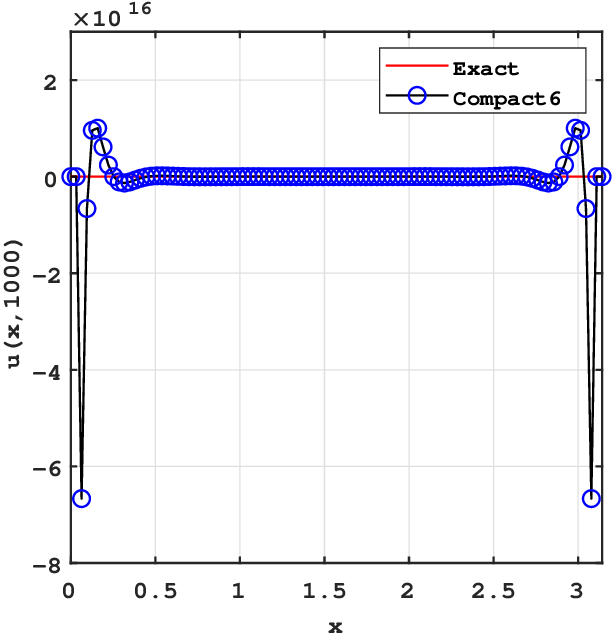}
      \subcaption{Unstable Solution, $\tau = 2.1$, $T = 1000$}
    \end{minipage}
        \begin{minipage}[b]{0.45\linewidth}
      \includegraphics[width=\linewidth]{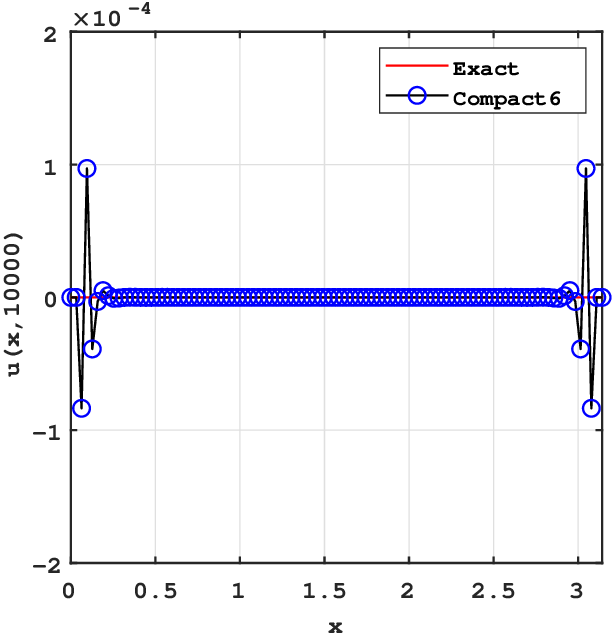}
      \subcaption{Stable Solution, $\tau = 2$, $T = 10,000$}
    \end{minipage}\hfill
    \begin{minipage}[b]{0.45\linewidth}
      \includegraphics[width=\linewidth]{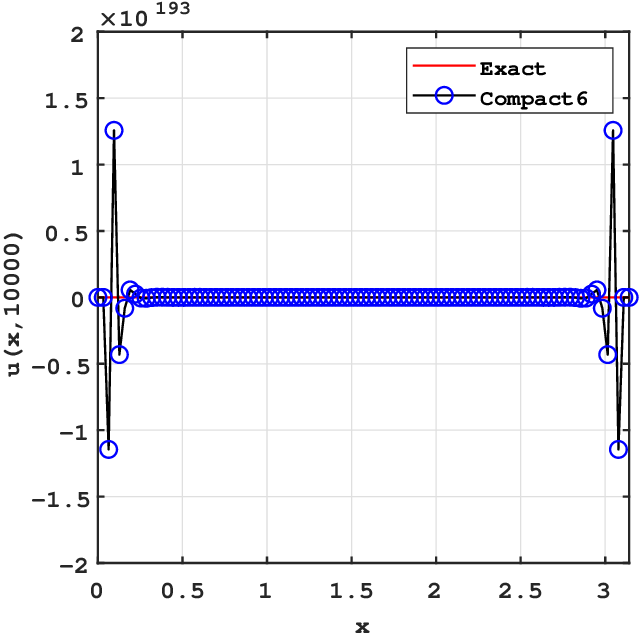}
      \subcaption{Unstable Solution, $\tau = 2.1$, $T = 10,000$}
    \end{minipage}
    \caption{Numerical solutions obtained using the Compact6 method for Example~\ref{example:1} with $N = 100$: (a) $\tau = 2$ and (b) $\tau = 2.1$ at $T = 1000$, and (c) $\tau = 2$ and (d) $\tau = 2.1$ at $T = 10,000$.}
    \label{Figure:E1c}
\end{figure}

\end{example}
\begin{example}
\normalfont
(Advection Free Flow) \\
We now consider the two-dimensional linear Sobolev-type equation
\begin{equation}\label{example:2D_1}
    u_t = u_{xx} + u_{yy} + u_{xxt} + u_{yyt}, \quad (x,y,t) \in [0,30] \times [0,30] \times [0, T],
\end{equation}
with initial condition $u(x,y,0) = \sin(x)\sin(y)$. The exact solution is given by $u(x,y,t) = e^{-2t/3}\sin(x)\sin(y)$ in the same domain. The numerical results at $T = 1$ over the spatial region $[0, 30] \times [0, 30]$ are summarized in Table~\ref{Table:2D_E1a}, showing $L^{\infty}$, $L^{1}$, and $L^{2}$ errors along with the observed convergence rates. The results confirm that the Compact6 method attains sixth-order accuracy. Figure~\ref{Figure:2D_E1a} shows the numerical solution at $T = 1$. Subfigure~(i) presents the surface plot of the solution, while subfigure~(ii) displays the corresponding contour plot. Both plots highlight the accuracy and resolution of the computed solution on a $320 \times 320$ grid.

\begin{table}[htbp!]
\centering
\captionof{table}{Errors and order of convergence for Example \ref{example:2D_1} (Here $\tau = h^6$).}
\setlength{\tabcolsep}{0pt}
\begin{tabular*}{\textwidth}{@{\extracolsep{\fill}} l *{8}{c} }
\toprule
 $\boldsymbol{\textbf{N}_x \times \textbf{N}_y}$ & $\boldsymbol{L^{\infty}}$\textbf{-error} & \textbf{Rate} & $\boldsymbol{L^{1}}$\textbf{-error} & \textbf{Rate} & $\boldsymbol{L^{2}}$\textbf{-error} & \textbf{Rate}\\
\midrule
 40 $\times$ 40 & 9.0298e-02 & - & 3.2364e-02 & -& 3.9275e-02 & - \\
80 $\times$ 80& 1.2381e-03 & 6.1885 & 4.2315e-04 & 6.2571& 5.2048e-04 & 6.2376 \\
160 $\times$ 160& 2.6617e-05 & 5.5396 & 6.3855e-06 & 6.0502& 7.8688e-06 & 6.0476 \\
320 $\times$ 320& 4.4120e-07 & 5.9148 & 9.7715e-08 & 6.0301& 1.2091e-07 & 6.0241 \\
640 $\times$ 640& 6.9475e-09 & 5.9888 & 1.5097e-09 & 6.0163& 1.8731e-09 & 6.0124 \\
1280 $\times$ 1280& 1.0840e-10 & 6.0020 & 2.3448e-11 & 6.0086& 2.9140e-11 & 6.0063 \\
\bottomrule
\end{tabular*}
\label{Table:2D_E1a}
\end{table}

\begin{figure}[htbp!]
    \centering
    \begin{minipage}[b]{0.45\linewidth}
      \includegraphics[width=\linewidth]{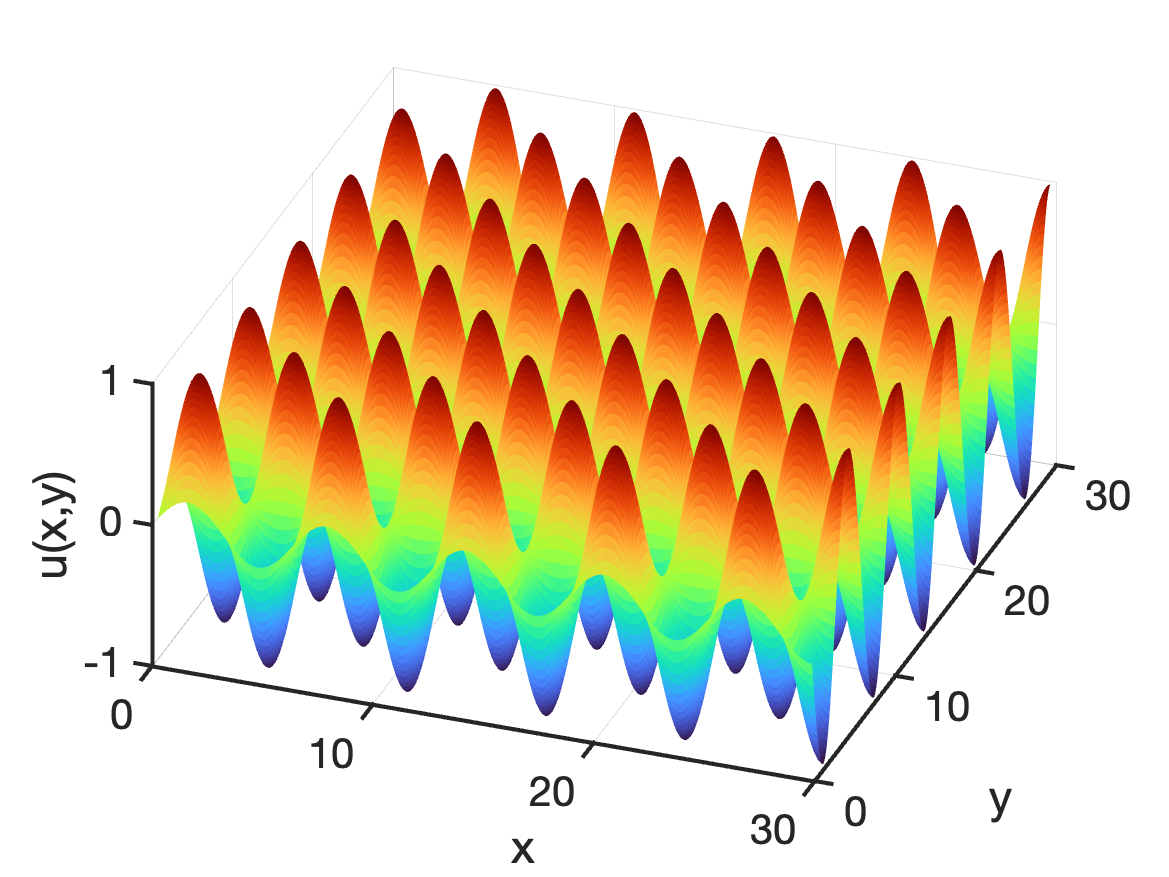}
      \subcaption*{(i) Numerical Solution}
    \end{minipage}\hfill
    \begin{minipage}[b]{0.45\linewidth}
      \includegraphics[width=\linewidth]{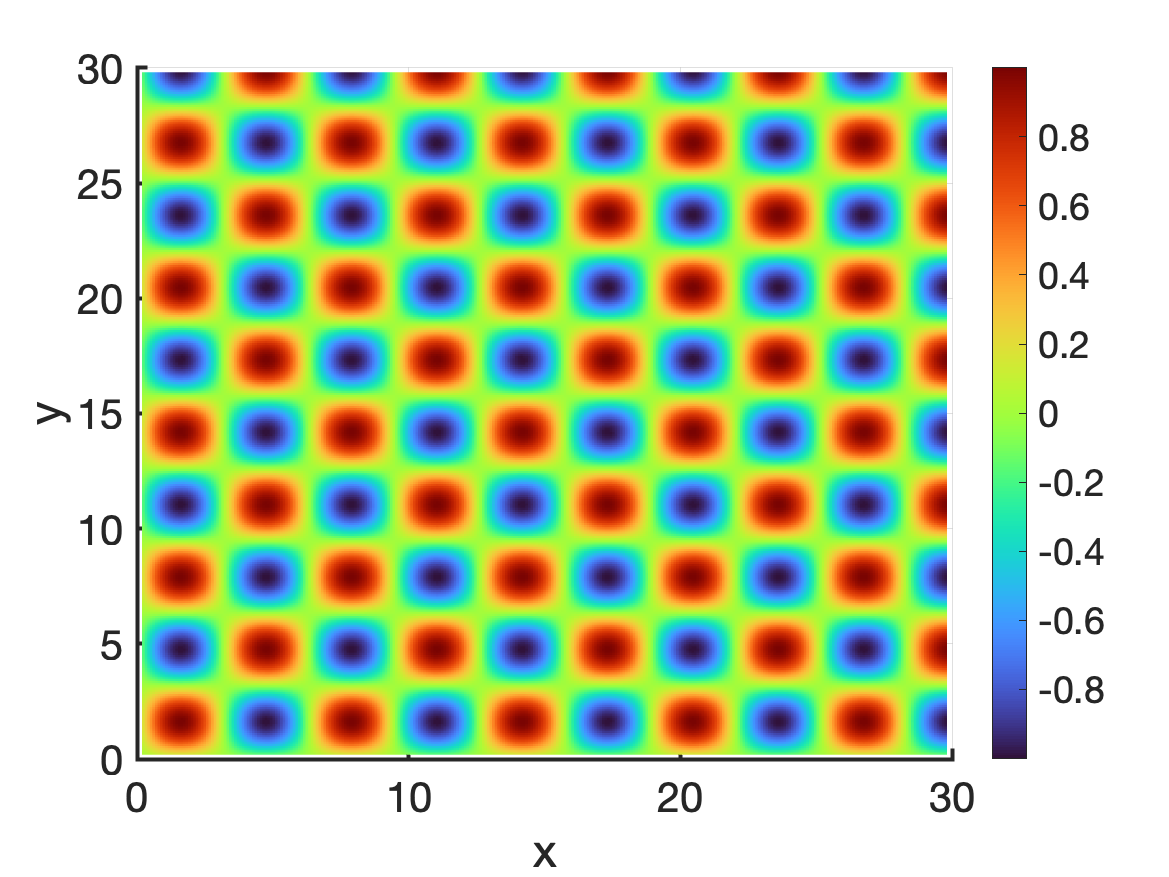}
      \subcaption*{(ii) Contour plot}
    \end{minipage}
    \caption{ Numerical solution and corresponding contour plot obtained using the Compact6 scheme Example~\ref{example:2D_1} at final time $T = 1$, with a grid resolution of ${N}_x \times {N}_y = 320 \times 320$.}
    \label{Figure:2D_E1a}
\end{figure} 
\end{example}
\begin{example}
\normalfont
(Advection-Diffusion Flow) \\
Consider the linear Sobolev-type equation
\begin{equation}\label{example:2}
    u_t + u_x = u_{xx} + u_{xxt}, \quad (x,t) \in [0,30] \times [0, 1],
\end{equation}
with the initial condition $u(x,0) = \sin(x)$. The exact solution is given by $u(x,t) = e^{-t/2}\sin(x-\dfrac{t}{2})$. Implementing the Compact6 method with boundary correction, we use the exact solution to the boundary terms, following the approach outlined in the previous example. In Table~\ref{Table:E2a}, we report the $L^{\infty}$, $L^{1}$, and $L^{2}$ errors, which indicate that the method achieves nearly sixth-order convergence. 
Table~\ref{Table:E2b} displays the temporal error and convergence results at the final time \( T = 1 \), using a fixed spatial discretization of \( N = 300 \). The numerical scheme demonstrates clear first-order accuracy in time across all norms, validating its temporal convergence behavior. Additionally, the comparative analysis in Fig.~\ref{Figure:E2d} highlights the $L^{\infty}$ errors of CBSQI, ICBSQI, and Compact6 schemes for different spatial resolutions with \( \tau = 10^{-4} \). Among these, the Compact6 method consistently outperforms the others, yielding significantly smaller errors and showcasing its superior accuracy and robustness on uniform grids.
\par
We then examine linear stability under the conditions specified by equation \ref{thm:dt1}, where the stability criterion for the Compact6 method is $\tau \leq \dfrac{2 \biggl(1 + \dfrac{7 h^2}{48} \biggr)}{\biggl(1+ \dfrac{507291 h^2}{9000000} \biggr)}$. For $h = 0.0314$ and a computational interval of $[0, \pi]$, this yields an approximate stability threshold of 2. As observed in prior examples, our numerical tests support this threshold as optimal as shown in Fig.~\ref{Figure:E2c}.

\begin{table}[htbp!]
\centering
\captionof{table}{Errors and order of convergence for Example \ref{example:2}.}
\setlength{\tabcolsep}{0pt}
\begin{tabular*}{\textwidth}{@{\extracolsep{\fill}} l *{8}{c} }
\toprule
 \textbf{N} & $\boldsymbol{L^{\infty}}$\textbf{-error} & \textbf{Rate} & $\boldsymbol{L^{1}}$\textbf{-error} & \textbf{Rate} & $\boldsymbol{L^{2}}$\textbf{-error} & \textbf{Rate}\\
\midrule
 40& 4.0010e-02 & - & 2.0771e-02 & - & 2.2165e-02 & - \\
 80& 5.6453e-04 & 6.1472 & 2.8795e-04 & 6.1726& 3.1281e-04 & 6.1468 \\
 160& 8.2482e-06 & 6.0968 & 4.2936e-06 & 6.0675& 4.7370e-06 & 6.0452 \\
 320& 1.2505e-07 & 6.0435 & 6.5461e-08 & 6.0354& 7.2949e-08 & 6.0209 \\
\bottomrule
\end{tabular*}\label{Table:E2a}
\end{table}
\begin{table}[htbp!]
\centering
\captionof{table}{Errors and order of convergence for Example \ref{example:2} at $T=1$ and $N=300$.}
\setlength{\tabcolsep}{0pt}
\begin{tabular*}{\textwidth}{@{\extracolsep{\fill}} l *{8}{c} }
\toprule
 \textbf{$\boldsymbol{\tau}$} & $\boldsymbol{L^{\infty}}$\textbf{-error} & \textbf{Rate} & $\boldsymbol{L^{1}}$\textbf{-error} & \textbf{Rate} & $\boldsymbol{L^{2}}$\textbf{-error} & \textbf{Rate}\\
\midrule
 1e-01& 1.8553e-02 & - & 9.7956e-03 & -& 1.0902e-02 & - \\
1e-02& 1.8173e-03 & 1.0090 & 9.5157e-04 & 1.0126& 1.0590e-03 & 1.0126 \\
1e-03& 1.8135e-04 & 1.0009 & 9.4883e-05 & 1.0013& 1.0560e-04 & 1.0012 \\
1e-04& 1.8131e-05 & 1.0001 & 9.4855e-06 & 1.0001& 1.0556e-05 & 1.0001 \\
1e-05& 1.8128e-06 & 1.0001 & 9.4836e-07 & 1.0001& 1.0555e-06 & 1.0001 \\
\bottomrule
\end{tabular*}
\label{Table:E2b}
\end{table}
 \begin{figure}[htbp!]  
    \centering
    \begin{minipage}[b]{0.45\linewidth}
      \includegraphics[width=\linewidth]{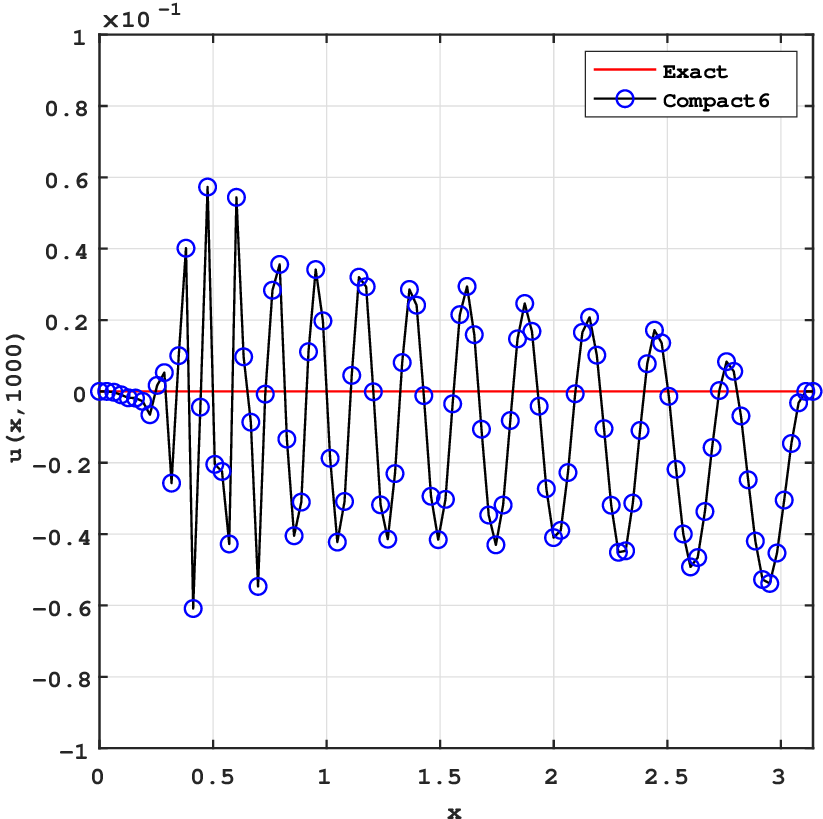}
      \subcaption*{Stable Solution, $\tau = 2$}
    \end{minipage}\hfill
    \begin{minipage}[b]{0.45\linewidth}
      \includegraphics[width=\linewidth]{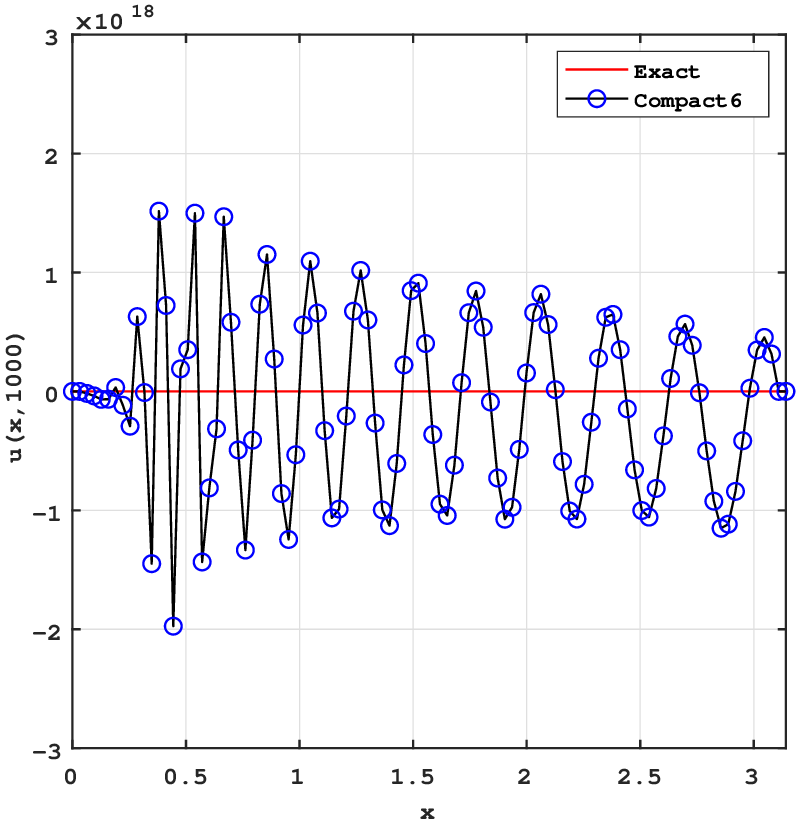}
      \subcaption*{Unstable Solution, $\tau = 2.1$}
    \end{minipage}
    \caption{Numerical solutions using the Compact6 method for Example in \ref{example:2} at $T = 1000$ with $N = 100$: (i) $\tau = 2$, (ii) $\tau = 2.1$.}
    \label{Figure:E2c}
\end{figure}

\begin{figure}[htbp!]  
    \centering
    \begin{minipage}[b]{0.45\linewidth}
      \includegraphics[width=\linewidth]{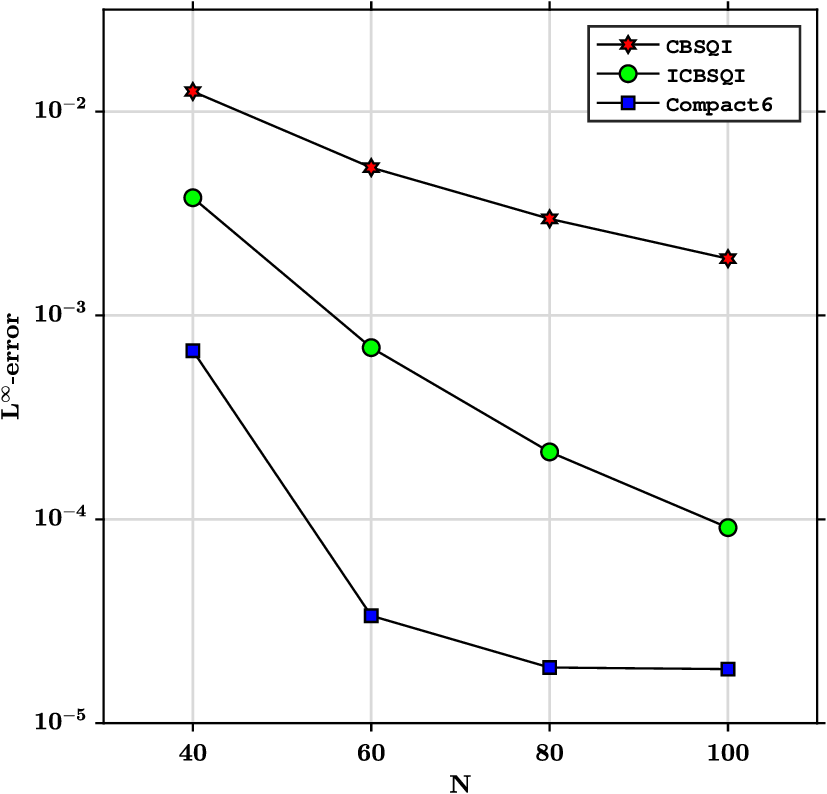}
    \end{minipage}\hfill
    \caption{ Comparison of CBSQI, ICBSQI and Compact6 schemes in terms of $L^{\infty}$ errors (in $log10$ scale) for Example~\ref{example:2} at $T=1$ and $\tau =0.0001$.}
    \label{Figure:E2d} 
\end{figure}

\end{example}
\begin{example}
\normalfont
(Advection-Diffusion Flow) \\
We extend the study to two dimensions by considering the 2D linear Sobolev-type equation
\begin{equation}\label{example:2D_2}
    u_t + u_x + u_y = u_{xx} + u_{yy} + u_{xxt} + u_{yyt}, \quad (x,y,t) \in [0,30] \times [0,30] \times [0, 1],
\end{equation}
subject to the initial condition $u(x,y,0) = \sin(x)\sin(y)$. The exact solution is given by
\[
u(x,y,t) = e^{-2t/3}\sin\left(x - \dfrac{t}{3}\right)\sin\left(y - \dfrac{t}{3}\right).
\]
The Compact6 method is employed to compute the solution, with boundary corrections extended from the 1D formulation. Table~\ref{Table:2D_E2a} presents the $L^{\infty}$, $L^{1}$, and $L^{2}$ error norms, demonstrating that the scheme maintains nearly sixth-order accuracy in two dimensions as well.
\begin{table}[htbp!]
\centering
\captionof{table}{Errors and order of convergence for Example \ref{example:2D_2} (Here $\tau = h^6$).}
\setlength{\tabcolsep}{0pt}
\begin{tabular*}{\textwidth}{@{\extracolsep{\fill}} l *{8}{c} }
\toprule
 $\boldsymbol{\textbf{N}_x \times \textbf{N}_y}$ & $\boldsymbol{L^{\infty}}$\textbf{-error} & \textbf{Rate} & $\boldsymbol{L^{1}}$\textbf{-error} & \textbf{Rate} & $\boldsymbol{L^{2}}$\textbf{-error} & \textbf{Rate}\\
\midrule
 40 $\times$ 40& 9.4893e-02 & - & 3.9219e-02 & - & 4.7519e-02 & - \\
80 $\times$ 80& 2.3417e-03 & 5.3407 & 5.3359e-04 & 6.1997& 6.4883e-04 & 6.1945 \\
160 $\times$ 160& 3.5648e-05 & 6.0376 & 8.1059e-06 & 6.0406& 9.9268e-06 & 6.0304 \\
320 $\times$ 320& 5.3646e-07 & 6.0542 & 1.2449e-07 & 6.0248& 1.5323e-07 & 6.0176 \\
640 $\times$ 640& 8.1970e-09 & 6.0322 & 1.9265e-09 & 6.0139& 2.3783e-09 & 6.0096 \\
1280 $\times$ 1280& 1.2659e-10 & 6.0168 & 2.9949e-11 & 6.0073& 3.7032e-11 & 6.0050 \\
\bottomrule
\end{tabular*}
\label{Table:2D_E2a}
\end{table}
\end{example}
\par
We proceed by applying the Compact6 method to examine certain nonlinear Sobolev-type equations, specifically focusing on two cases: the EW equation and the BBMB equation. The EW equation is an advection-only model, neglecting any dissipation effects, while the BBMB equation incorporates advection, diffusion, and dispersion processes.
\subsection{Equal width equation}
Similar to the RLW equation, the EW equation is characterized by three conserved quantities, representing mass, momentum, and energy, respectively. Olver \cite{olver1979euler} demonstrated that the EW equation can be reformulated in a conservative form $\eta_t + \chi_x = 0$ in only three distinct, non-trivial ways, where $\eta = \eta(u, u_x)$ and $\chi = \chi(u, u_t, u_{xt})$. These invariants are given by
\begin{equation}\label{Invariants}
    \mathfrak{I}_1 = \int_{-\infty}^{\infty} u \, dx, \quad \mathfrak{I}_2 = \int_{-\infty}^{\infty} \left( u^2 + \delta u_x^2 \right) dx, \quad \mathfrak{I}_3 = \int_{-\infty}^{\infty} u^3 \, dx.
\end{equation}
To approximate these integrals, we utilized the Composite Simpson’s rule. This section investigates four test cases to assess the effectiveness of the proposed numerical method: solitary wave solutions, interactions between two solitary waves, interactions of three solitary waves, and the undular bore phenomenon.
\begin{example}\label{example:3}
\normalfont
(Propagation of a single solitary wave) \\
Consider the equal width equation (\ref{eqn:EW}) with the initial condition
\begin{equation}\label{IC1:3}
    u(x,0) = 3c sech^2(k(x-x_0)), \quad x \in [0,30],
\end{equation}
along with the homogeneous Dirichlet boundary conditions
\begin{equation}
    u(0,t) = 0 \quad \text{and} \quad u(30,t) = 0.
\end{equation}
 The exact solution is given by 
\begin{equation}\label{IC1:3_Exact}
    u(x,t) = 3c sech^2\Bigl(k(x-x_0-ct)\Bigr).
\end{equation}
This solution represents the movement of a solitary wave with an amplitude of $3c$ and a width of $k$, initially centered at $x=x_0$, where $k=\sqrt{\dfrac{1}{4\delta}}$ and $c$ denotes the wave speed. Here, $\delta = 1$, $c=0.03$, $x_0 = 10$. Table~\ref{Table:E3a} provides a summary of the $L^{\infty}$-error, $L^{1}$-error, $L^{2}$-error, and the order of convergence for considered scheme at the final time $T = 200$. Fig.~\ref{Figure:E3a}(a) illustrates the absolute error, while Fig.~\ref{Figure:E3a}(b) compares the numerical solution obtained using the Compact6 method with the exact solution at times $T = 0$ and $T = 200$. 
Fig.~\ref{Figure:E3b} provides a comparative study of the $L^{\infty}$ errors at final time \( T = 200 \), with a fixed time step size \( \tau = 10^{-4} \). The results clearly indicate that the Compact6 scheme achieves significantly lower errors compared to the CBSQI and ICBSQI methods across all tested spatial resolutions. This demonstrates the high accuracy and stability of the Compact6 approach, especially for long-time integration.
\par
For solitary wave solutions (\ref{IC1:3_Exact}), each form of $\chi$ satisfies
\begin{equation}\label{Invariant:condition}
    \chi \to 0 \quad \text{as} \quad x \to \pm \infty,
\end{equation}
implying that the integrals
\begin{equation}
    \mathfrak{I}_j = \int_{-\infty}^{\infty} \eta_j \, dx, \quad j = 1, 2, 3,
\end{equation}
remain invariant over time. Given that the solution in this example is smooth, it must satisfy all three conservation laws. Thus, the invariants $\mathfrak{I}_1$, $\mathfrak{I}_2$, and $\mathfrak{I}_3$ remain constant for the solitary wave solution, serving as a benchmark for evaluating the conservation properties of numerical methods. To validate this property, we conducted numerical experiments, limiting the domain of integration to $[0, 30]$ due to the decay of the solitary wave solution outside this interval. The analytical expressions for these invariants, provided in \cite{gardner1997simulations}, are given by
\begin{equation}
    \mathfrak{I}_1 = \dfrac{6c}{k}, \quad \mathfrak{I}_2 = \dfrac{12c^2}{k} + \dfrac{48kc^2 \delta}{5}, \quad \mathfrak{I}_3 = \dfrac{144c^3}{5k}.
\end{equation}
For $c=0.03$, the exact values of $\mathfrak{I}_1$, $\mathfrak{I}_2$, and $\mathfrak{I}_3$ are $0.36$, $0.02592$, and $1.5552 \times 10^{-3}$, respectively. Table~\ref{Table:E3b} summarizes the numerical values obtained for these invariants and their respective percentage errors, showing that $\mathfrak{I}_1$, $\mathfrak{I}_2$, and $\mathfrak{I}_3$ remain effectively constant throughout the simulation. At $t=25$, the relative changes in the invariants $\mathfrak{I}_1$, $\mathfrak{I}_2$, and $\mathfrak{I}_3$ for this method are 2.3237e-03\%, 1.4909e-04\%, and 2.3542e-04\%, respectively. The close agreement between the numerical results and the analytical values supports the time invariance of these integrals, validating the conservation properties of the method used.
\begin{table}[htbp!]
\centering
\captionof{table}{Errors and order of convergence for Example \ref{example:3}}
\setlength{\tabcolsep}{0pt}
\begin{tabular*}{\textwidth}{@{\extracolsep{\fill}} l *{8}{c} }
\toprule
 \textbf{N} & $\boldsymbol{L^{\infty}}$\textbf{-error} & \textbf{Rate} & $\boldsymbol{L^{1}}$\textbf{-error} & \textbf{Rate} & $\boldsymbol{L^{2}}$\textbf{-error} & \textbf{Rate}\\
\midrule
 40& 5.9911e-04 & - & 1.1793e-04 & - & 1.8701e-04 & - \\
 80& 8.6472e-06 & 6.1144 & 1.6380e-06 & 6.1698& 2.7540e-06 & 6.0855 \\
 160& 1.3224e-07 & 6.0310 & 2.4203e-08 & 6.0807& 4.1772e-08 & 6.0428 \\
\bottomrule
\end{tabular*}\label{Table:E3a}
\end{table}

\begin{table}[htbp!]
\centering
\captionof{table}{Invariants for a single solitary wave in Example \ref{example:3} over the domain $\Omega = [0, 30]$ with $N=120$.}
\setlength{\tabcolsep}{0pt}
\begin{tabular*}{\textwidth}{@{\extracolsep{\fill}} l *{8}{c} }
\toprule
\textbf{Time} &
$\mathfrak{I}_1$ & $\%$ Error of $\mathfrak{I}_1$ & $\mathfrak{I}_2$ & $\%$ Error of $\mathfrak{I}_2$ & $\mathfrak{I}_3$ & $\%$ Error of $\mathfrak{I}_3$\\
\midrule
  5& 3.5998e-01 & 4.2334e-03 & 2.5920e-02 & 2.6234e-05 & 1.5552e-03 & 4.7084e-05\\
 10& 3.5999e-01 & 3.6438e-03 & 2.5920e-02 & 5.6620e-05 & 1.5552e-03 & 9.4169e-05\\
 15& 3.5999e-01 & 3.1363e-03 & 2.5920e-02 & 8.7267e-05 & 1.5552e-03 & 1.4125e-04\\
 20& 3.5999e-01 & 2.6996e-03 & 2.5920e-02 & 1.1811e-04 & 1.5552e-03 & 1.8834e-04\\
 25& 3.5999e-01 & 2.3237e-03 & 2.5920e-02 & 1.4909e-04 & 1.5552e-03 & 2.3542e-04\\
\bottomrule
\end{tabular*}\label{Table:E3b}
\end{table}

\begin{figure}[htbp!]
  \begin{minipage}[b]{0.45\linewidth}
    \centering
    \includegraphics[trim=0cm 0cm 0cm 0cm, clip=true, width=\linewidth]{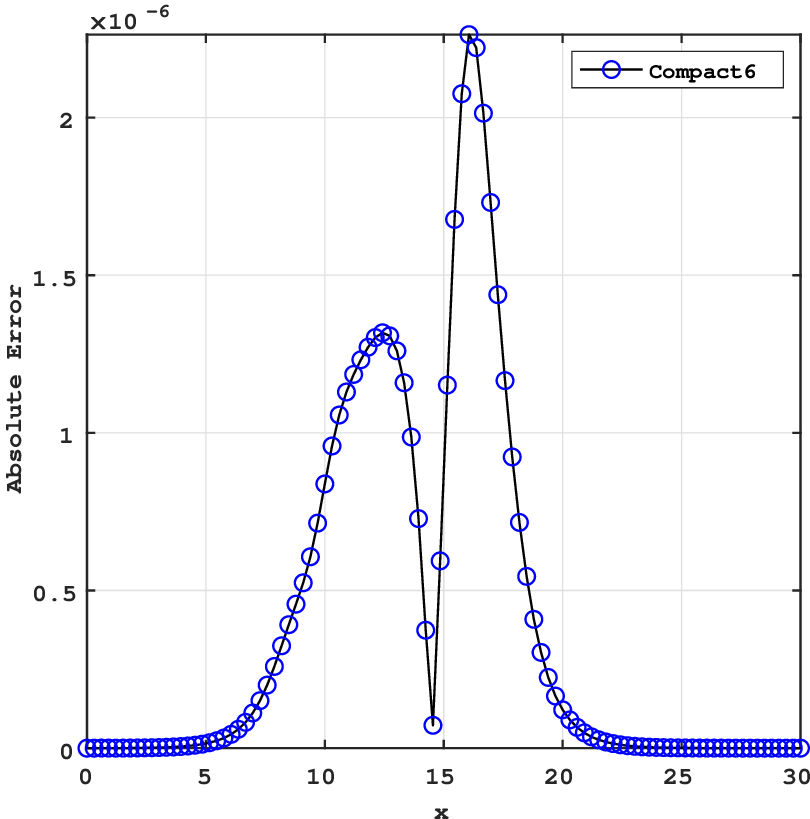}
    \captionsetup{justification=centering}
    \subcaption*{\normalsize{(a) Absolute Error Distribution}}
  \end{minipage}\hfill
  \begin{minipage}[b]{0.45\linewidth}
    \centering
    \includegraphics[trim=0cm 0cm 0cm 0cm, clip=true, width=\linewidth]{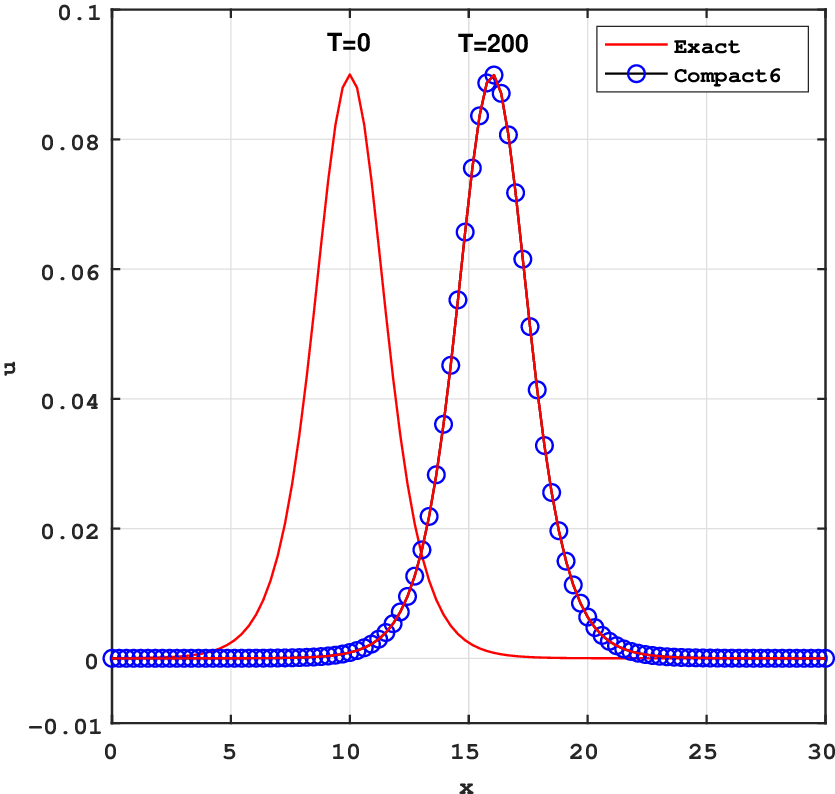}
    \captionsetup{justification=centering}
    \subcaption*{\normalsize{(b) Numerical vs. Exact Solution}}
  \end{minipage}
  \caption{(a) Absolute error distribution for the Compact6 method in Example \ref{example:3} at $T = 200$. (b) Comparison of the numerical solution from the Compact6 method (o symbols) with the exact solution (solid line) at $T = 200$ with $N=100$.}\label{Figure:E3a}
\end{figure}

\begin{figure}[htbp!]  
    \centering
    \begin{minipage}[b]{0.45\linewidth}
      \includegraphics[width=\linewidth]{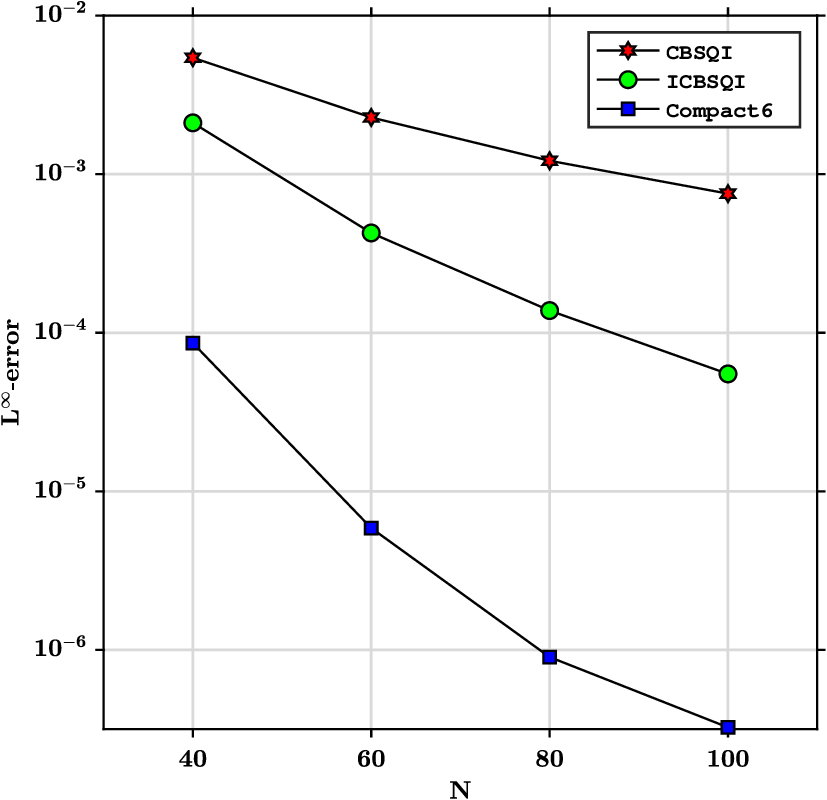}
    \end{minipage}\hfill
    \caption{ Comparison of CBSQI, ICBSQI and Compact6 schemes in terms of $L^{\infty}$ errors (in $log10$ scale) for Example~\ref{example:3} at $T=200$ and $\tau =0.0001$.}
    \label{Figure:E3b} 
\end{figure}
\end{example}
\begin{example}\label{example:4}
\normalfont
(Interaction of two solitary waves) \\
To analyze the interaction of two solitary waves, we consider Eq. (\ref{eqn:EW}) with the following initial and boundary conditions.
\begin{equation}\label{IC1:4}
\begin{split}
     u(x,0) &=3 \sum_{j=1}^2 c_j sech^2\Bigl(k_j(x-x_j)\Bigr), \quad x \in [0,70]\\
     u(0,t) &= 0 \quad \text{and} \quad u(70,t) = 0.
\end{split}
\end{equation}
For the numerical simulations, the parameters are set as  $k_1=k_2=0.5$, $\delta = 1$, $c_1=0.2$, $c_2=0.4$, $x_1 = 10$, $x_2 = 25$. The analytical values \cite{raslan2004computational} of the three invariants in (\ref{Invariants}) can be calculated as follows:
\begin{equation}
    \mathfrak{I}_1 = 12(c_1 + c_2) = 7.2 , \quad \mathfrak{I}_2 = 28.8(c_1^2 + c_2^2) = 5.76, \quad \mathfrak{I}_3 = 57.6(c_1^3 + c_2^3) = 4.1472.
\end{equation}
Table~\ref{Table:E4} presents the invariant values calculated at time points representing the pre-interaction phase ($T = 10$), interaction stages ($T = 45$, 55, and 65), and post-interaction phase ($T = 100$). The relative changes in the invariants  $\mathfrak{I}_1$, $\mathfrak{I}_2$, and $\mathfrak{I}_3$  at $t=25$ for the current method are 3.3634e-03\%, 5.5535e-02\%, and 8.9002e-02\%, respectively. The solution profiles at various time instances are shown in Fig.~\ref{Figure:E4a}. These graphs illustrate that, over time, the wave with the larger amplitude overtakes and moves ahead of the smaller amplitude wave.

\begin{table}[htbp!]
\centering
\captionof{table}{Invariants for the interaction of two solitary waves in Example \ref{example:4} over the domain $\Omega = [0, 70]$ with $N=300$.}
\setlength{\tabcolsep}{0pt}
\begin{tabular*}{\textwidth}{@{\extracolsep{\fill}} l *{8}{c} }
\toprule
\textbf{Time} &
$\mathfrak{I}_1$ & $\%$ Error of $\mathfrak{I}_1$ & $\mathfrak{I}_2$ & $\%$ Error of $\mathfrak{I}_2$ & $\mathfrak{I}_3$ & $\%$ Error of $\mathfrak{I}_3$\\
\midrule
 10& 7.2002 & 3.2146e-03 & 5.7604 & 6.5176e-03 & 4.1476 & 1.0761e-02\\
 45& 7.2002 & 3.3619e-03 & 5.7616 & 2.8086e-02 & 4.1491 & 4.5583e-02\\
 55& 7.2002 & 3.3632e-03 & 5.7619 & 3.2720e-02 & 4.1494 & 5.2814e-02\\
 65& 7.2002 & 3.3644e-03 & 5.7621 & 3.5805e-02 & 4.1496 & 5.7325e-02\\
 100& 7.2002 & 3.3634e-03 & 5.7632 & 5.5535e-02 & 4.1509 & 8.9002e-02\\
 \bottomrule
\end{tabular*}\label{Table:E4}
\end{table}

\begin{figure}[htbp!]
  \begin{minipage}[b]{0.45\linewidth}
    \centering
    \includegraphics[trim=0cm 0cm 0cm 0cm, clip=true, width=\linewidth]{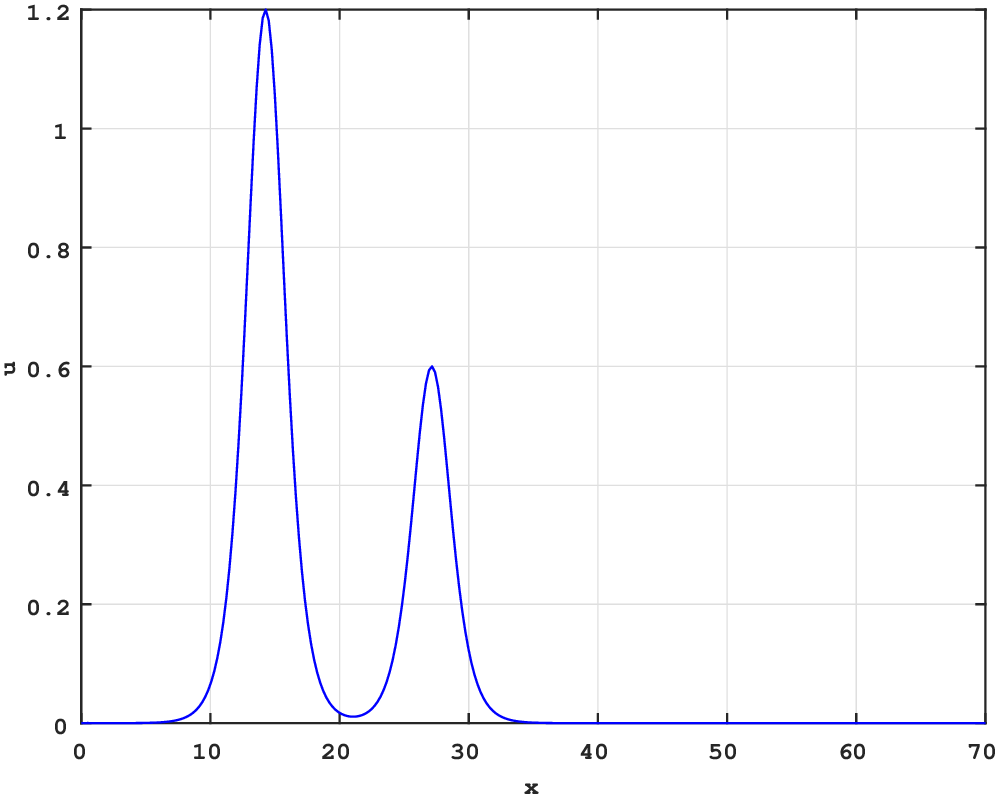}
    \captionsetup{justification=centering}
    \subcaption*{\normalsize{(a) $T=10$}}
  \end{minipage}\hfill
  \begin{minipage}[b]{0.45\linewidth}
    \centering
    \includegraphics[trim=0cm 0cm 0cm 0cm, clip=true, width=\linewidth]{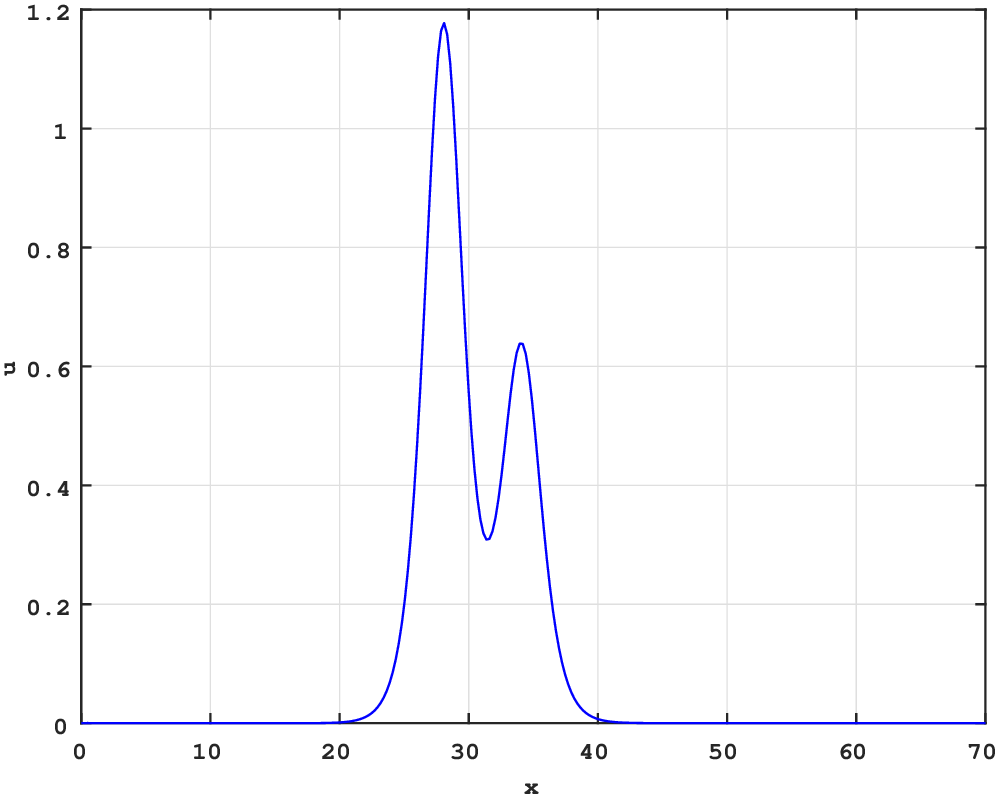}
    \captionsetup{justification=centering}
    \subcaption*{\normalsize{(b) $T=45$}}
  \end{minipage}
  \begin{minipage}[b]{0.45\linewidth}
    \centering
    \includegraphics[trim=0cm 0cm 0cm 0cm, clip=true, width=\linewidth]{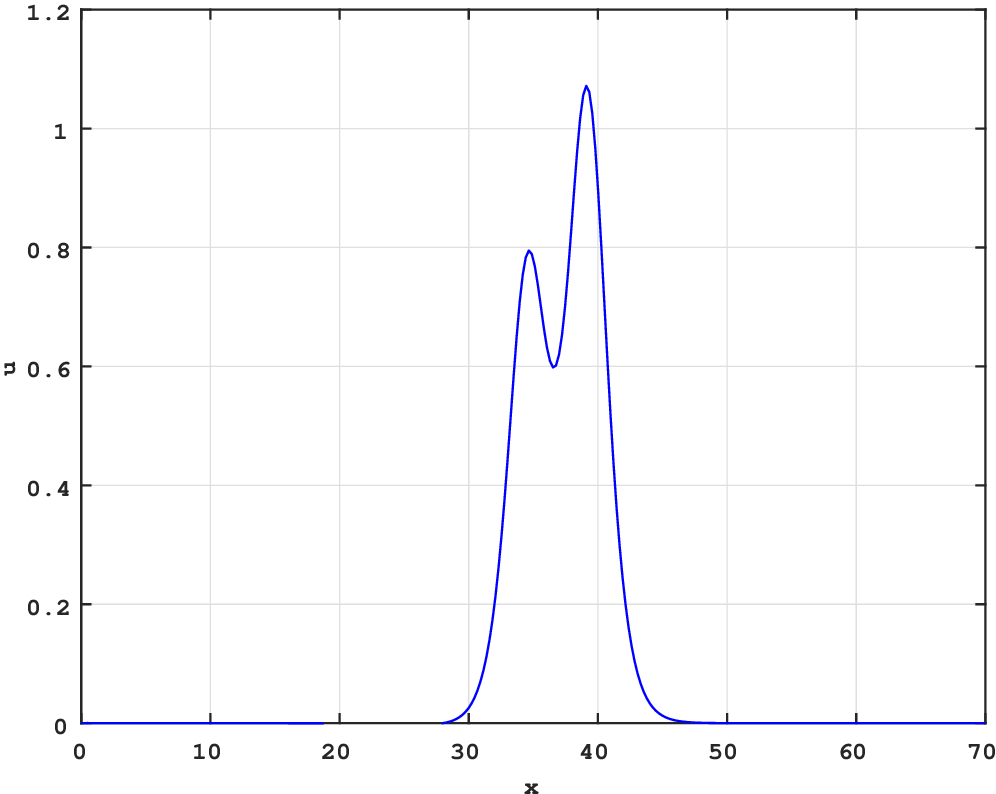}
    \captionsetup{justification=centering}
    \subcaption*{\normalsize{(c) $T=65$}}
  \end{minipage}\hfill
  \begin{minipage}[b]{0.45\linewidth}
    \centering
    \includegraphics[trim=0cm 0cm 0cm 0cm, clip=true, width=\linewidth]{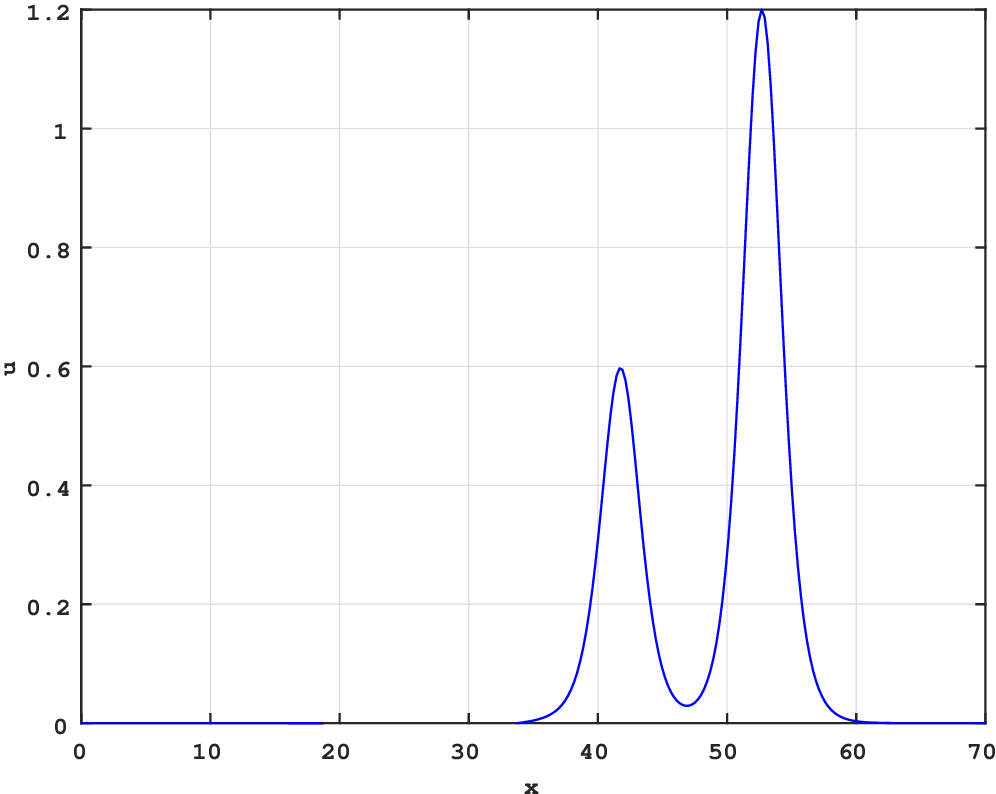}
    \captionsetup{justification=centering}
    \subcaption*{\normalsize{(d) $T=100$}}
  \end{minipage}
  \begin{minipage}[b]{0.45\linewidth}
    \centering
    \includegraphics[trim=0cm 0cm 0cm 0cm, clip=true, width=\linewidth]{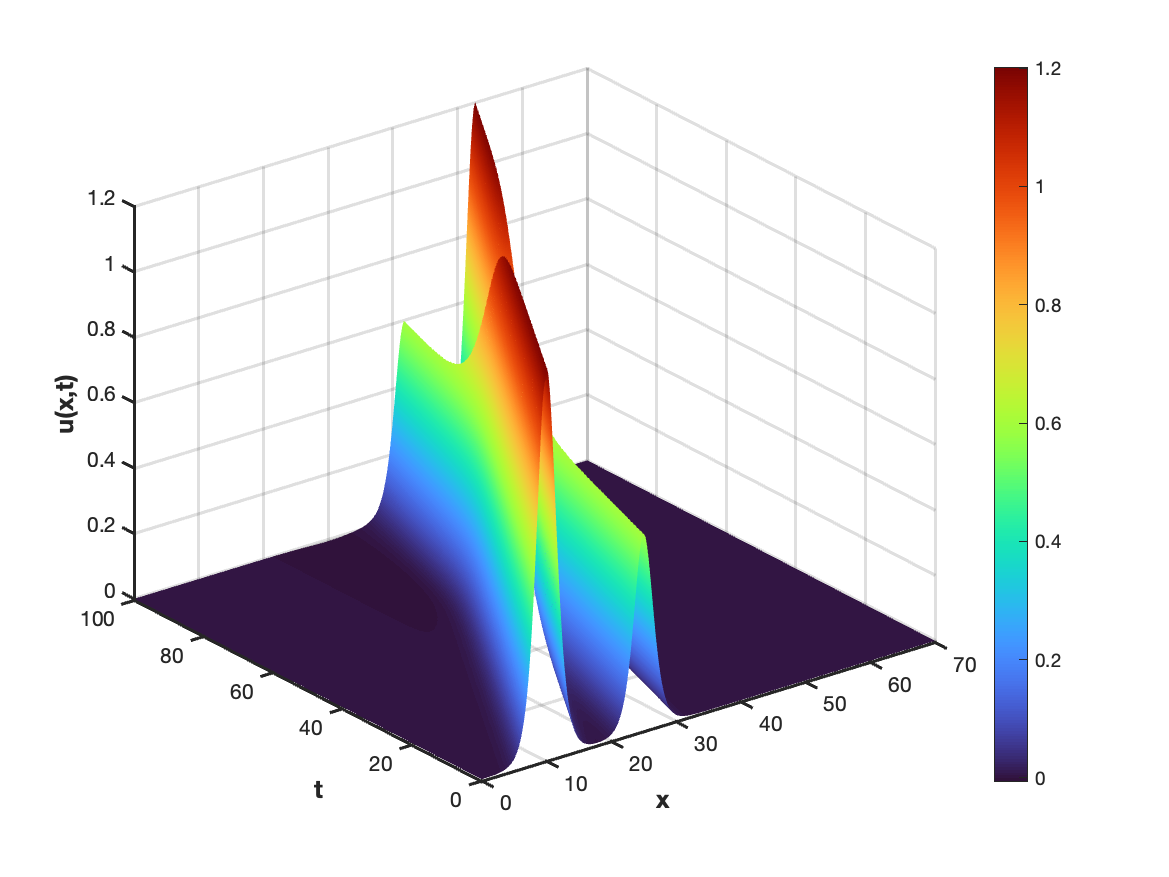}
    \captionsetup{justification=centering}
    \subcaption*{\normalsize{(e) Wave simulation across various time}}
  \end{minipage}\hfill
  \begin{minipage}[b]{0.45\linewidth}
    \centering
    \includegraphics[trim=0cm 0cm 0cm 0cm, clip=true, width=\linewidth]{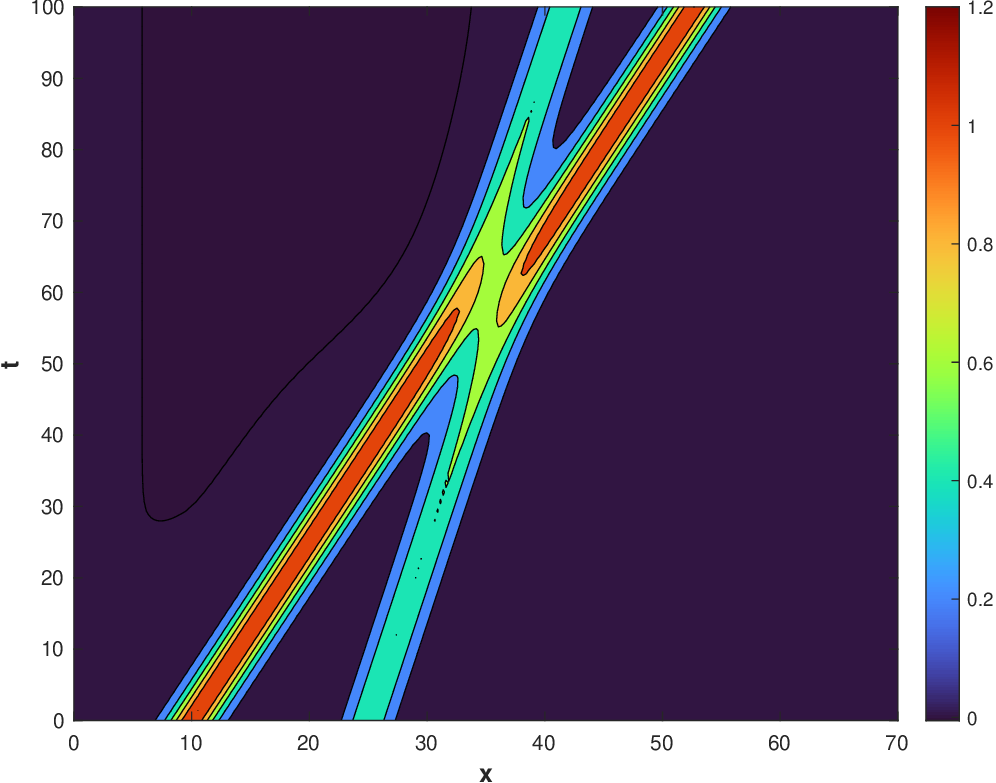}
    \captionsetup{justification=centering}
    \subcaption*{\normalsize{(f) Contour for wave simulation}}
  \end{minipage}
  \caption{Numerical solution of EW equation with $N=300$ and $0 \leq T\leq 100$ for Example \ref{example:4}.}\label{Figure:E4a}
\end{figure}
\end{example}
\begin{example}\label{example:4b}
\normalfont
(Interaction of three solitary waves) \\
The initial condition for Eq. (\ref{eqn:EW}) is defined as:
\begin{equation}\label{IC1:4b}
     u(x,0) =3 \sum_{j=1}^3 c_j sech^2\Bigl(k_j(x-x_j)\Bigr), \quad x \in [-10,100],
\end{equation}
which characterizes a combination of three solitary waves traveling in the same direction. We impose the boundary conditions \( u(a, 0) = 0 \) and $ u(b, 0) = 0 $, with $ a = -10 $ and $ b = 100 $. The time interval is taken as $ t \in [0, 15] $, and we use $ N = 600 $ evenly spaced points across this spatial domain. For the parameters, we set $ k_i = 0.5 $ (for $ i = 1, 2, 3 $), $ c_1 = 4.5 $, $ c_2 = 1.5 $, $ c_3 = 0.5 $, and positions $ x_1 = 10 $, $ x_2 = 25 $, $ x_3 = 35 $. Calculated values for the three conserved quantities are listed in Table~\ref{Table:E5}, showing that they closely match their theoretical values. Specifically, the invariants are given as follows:
\begin{equation}
    \mathfrak{I}_1 = 12(c_1 + c_2 + c_3) = 78 , \quad \mathfrak{I}_2 = 28.8(c_1^2 + c_2^2 + c_3^2) = 655.2, \quad \mathfrak{I}_3 = 57.6(c_1^3 + c_2^3 + c_3^3) = 5450.4.
\end{equation}
The results confirm that the invariants are well-preserved. The interactions of the three solitary waves are depicted in Fig.~\ref{Figure:E5}.
\begin{table}[htbp!]
\centering
\captionof{table}{Invariants for the interaction of three solitary waves in Example \ref{example:4b} over the domain $\Omega = [-10, 100]$ with $N=600$.}
\setlength{\tabcolsep}{0pt}
\begin{tabular*}{\textwidth}{@{\extracolsep{\fill}} l *{8}{c} }
\toprule
\textbf{Time} &
$\mathfrak{I}_1$ & $\%$ Error of $\mathfrak{I}_1$ & $\mathfrak{I}_2$ & $\%$ Error of $\mathfrak{I}_2$ & $\mathfrak{I}_3$ & $\%$ Error of $\mathfrak{I}_3$\\
\midrule
  1& 78.0000 & 2.0158e-07 & 655.3505 & 2.2971e-02 & 5452.2761 & 3.4422e-02\\
 5& 78.0000 & 2.0545e-07 & 655.7636 & 8.6019e-02 & 5457.9016 & 1.3763e-01\\
 10& 78.0000 & 2.0545e-07 & 656.3525 & 1.7591e-01 & 5465.9593 & 2.8547e-01\\
 15& 78.0000 & 8.7733e-08 & 657.0034 & 2.7525e-01 & 5474.6982 & 4.4581e-01\\
 \bottomrule
\end{tabular*}\label{Table:E5}
\end{table}

\begin{figure}[htbp!]
  \begin{minipage}[b]{0.45\linewidth}
    \centering
    \includegraphics[trim=0cm 0cm 0cm 0cm, clip=true, width=\linewidth]{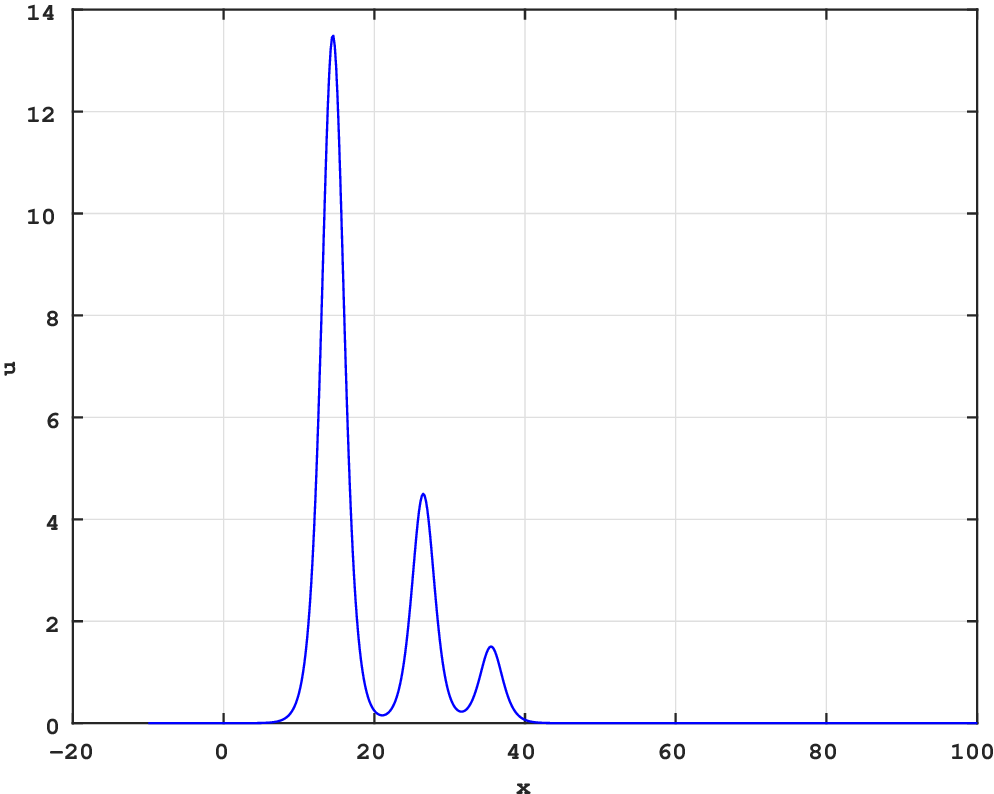}
    \captionsetup{justification=centering}
    \subcaption*{\normalsize{(a) $T=1$}}
  \end{minipage}\hfill
  \begin{minipage}[b]{0.45\linewidth}
    \centering
    \includegraphics[trim=0cm 0cm 0cm 0cm, clip=true, width=\linewidth]{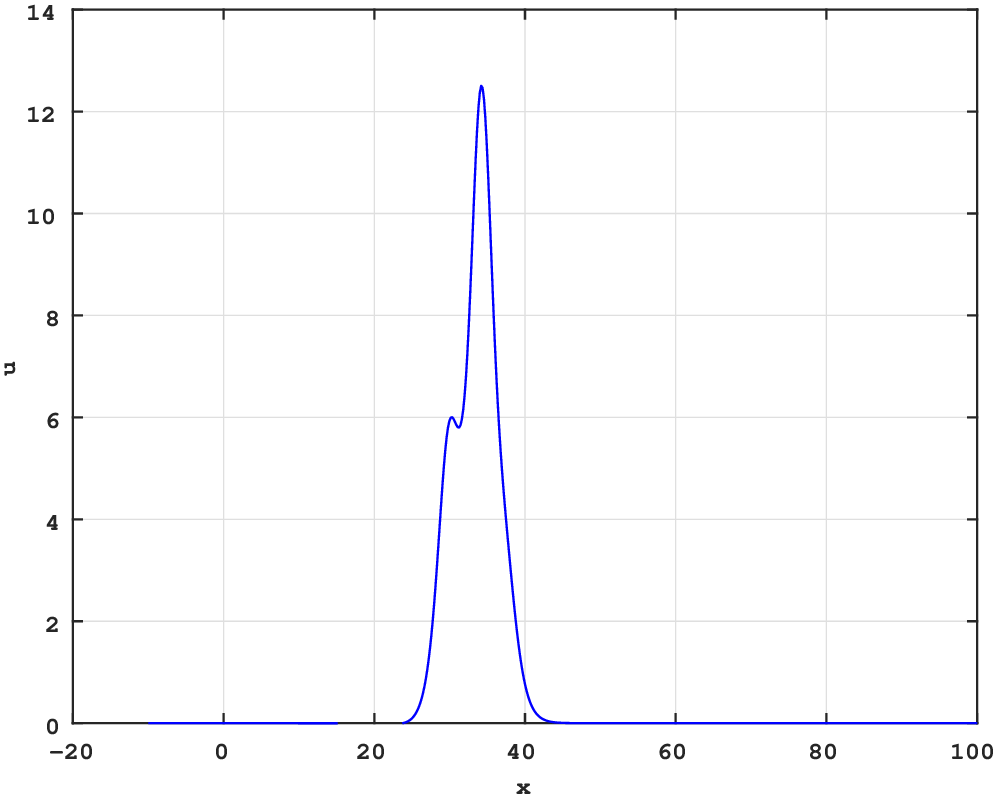}
    \captionsetup{justification=centering}
    \subcaption*{\normalsize{(b) $T=5$}}
  \end{minipage}
  \begin{minipage}[b]{0.45\linewidth}
    \centering
    \includegraphics[trim=0cm 0cm 0cm 0cm, clip=true, width=\linewidth]{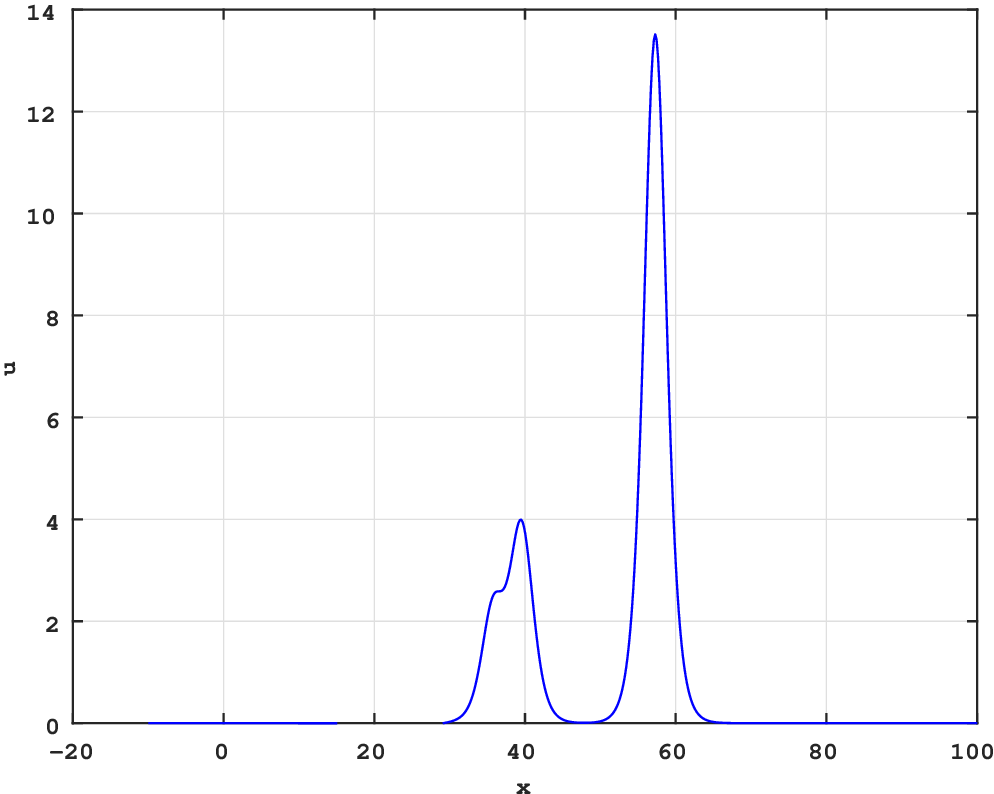}
    \captionsetup{justification=centering}
    \subcaption*{\normalsize{(c) $T=10$}}
  \end{minipage}\hfill
  \begin{minipage}[b]{0.45\linewidth}
    \centering
    \includegraphics[trim=0cm 0cm 0cm 0cm, clip=true, width=\linewidth]{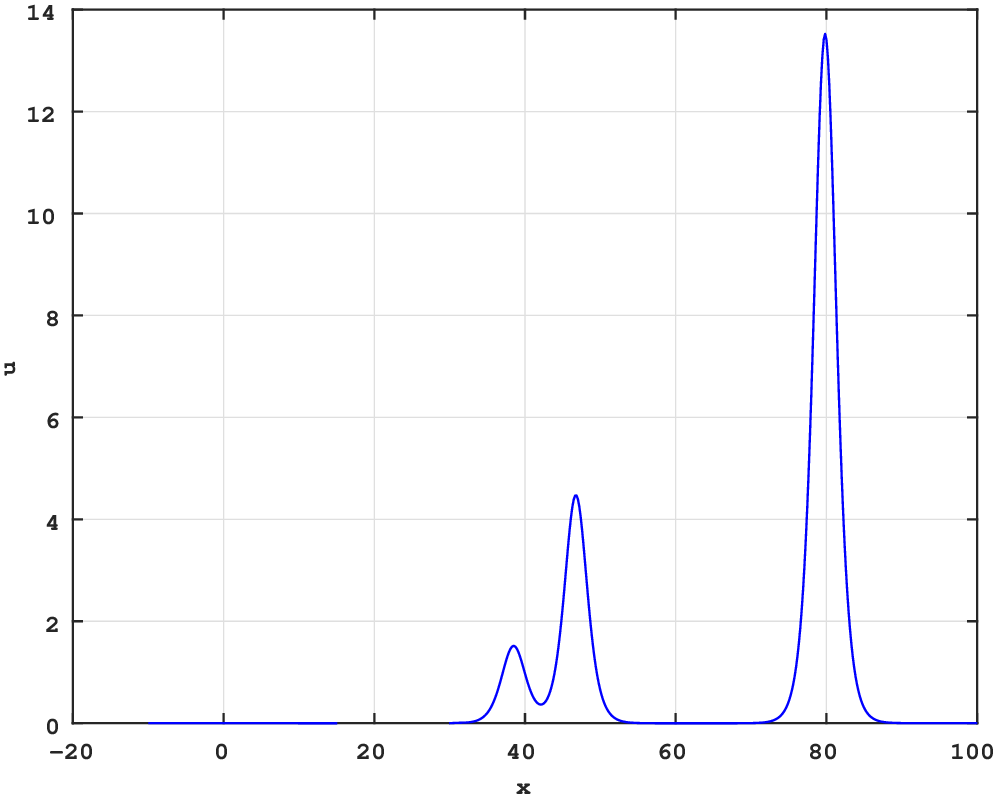}
    \captionsetup{justification=centering}
    \subcaption*{\normalsize{(d) $T=15$}}
  \end{minipage}
  \begin{minipage}[b]{0.45\linewidth}
    \centering
    \includegraphics[trim=0cm 0cm 0cm 0cm, clip=true, width=\linewidth]{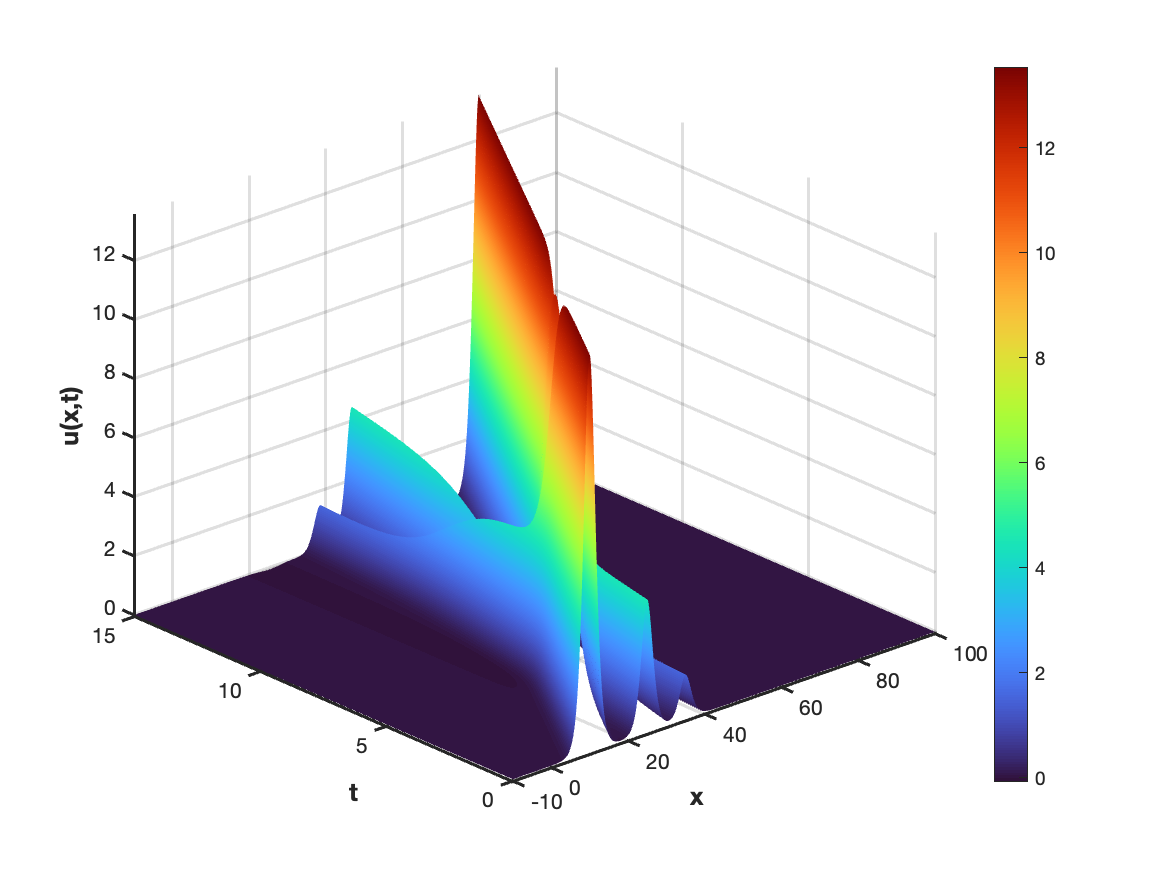}
    \captionsetup{justification=centering}
    \subcaption*{\normalsize{(e) Wave simulation across various time}}
  \end{minipage}\hfill
  \begin{minipage}[b]{0.45\linewidth}
    \centering
    \includegraphics[trim=0cm 0cm 0cm 0cm, clip=true, width=\linewidth]{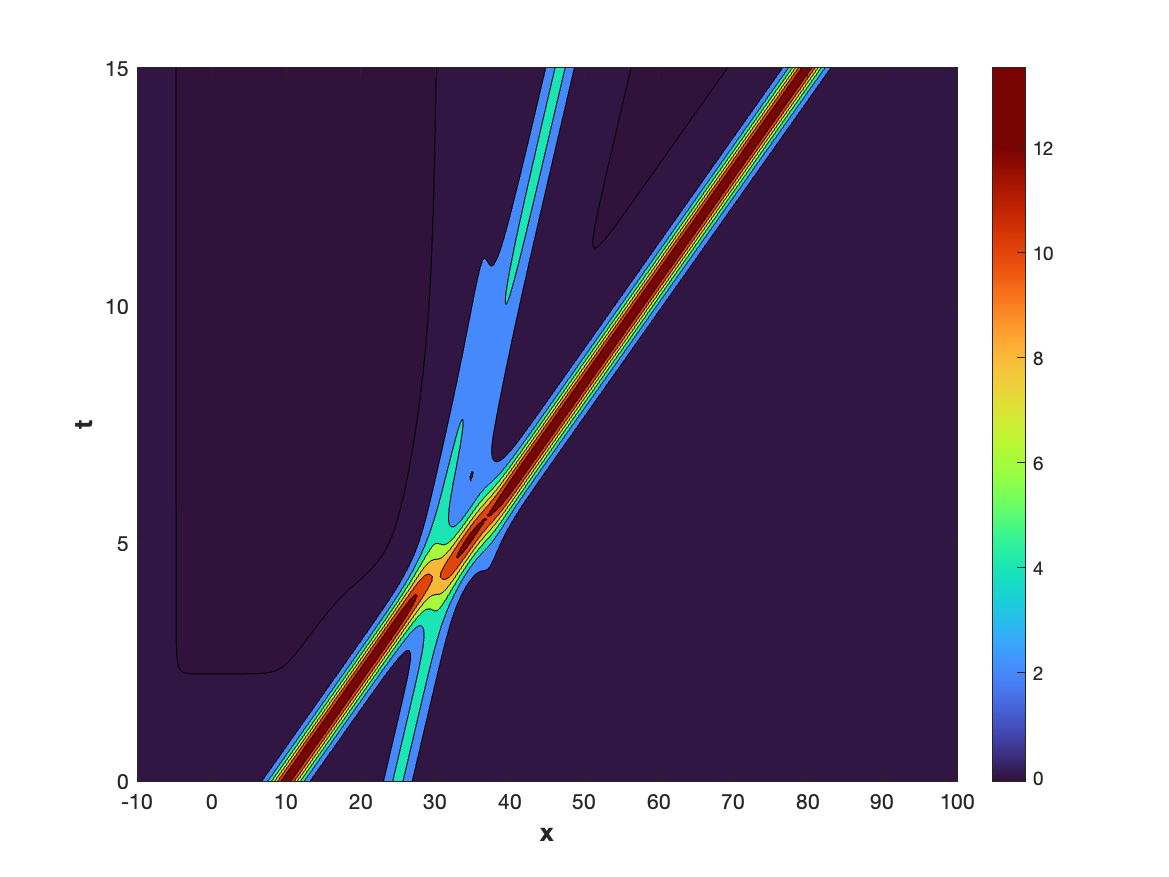}
    \captionsetup{justification=centering}
    \subcaption*{\normalsize{(f) Contour for wave simulation}}
  \end{minipage}
  \caption{Numerical solution of EW equation with $N=300$ and $0 \leq T\leq 15$ for Example \ref{example:4b}.}\label{Figure:E5}
\end{figure}
\end{example}
\begin{example}
\normalfont
(Undular Bore) \\
Next, the formation of an undular bore is examined using the following initial function, with the boundary condition 
\begin{align}\label{example:5}
    u(x,0) &= \dfrac{1}{2} u_0 \biggl(1 - \tanh \biggl(\dfrac{x - x_c}{d}\biggr)\biggr), \\
    u &\to 0 \text{ as } x \to \infty \text{ and } u \to u_0 \text{ as } x \to -\infty.
\end{align}
Here, $u(x,0)$ indicates the elevation of the water surface above the equilibrium position at $t=0$, the constants $u_0$, $d$, and $x_c$  represent the magnitude of the water level change, the slope between shallow and deep water, and the center of the disturbance, respectively. In this case, we set $u_0=0.1$, $d = 5$, $x_c=0$. The numerical solution is computed within the domain $-20 \leq x \leq 80$. As the wave moves forward, the nonlinear term $ u u_x $ in equation (\ref{eqn:EW}) causes a steepening effect, leading to an increase in $ u_x $. This behavior is evident in the numerical results at $ T = 70 $, calculated with the Compact6 method, as illustrated in Fig.~\ref{Figure:E6}(a). When the dispersive term is absent, the wave forms a shock at the critical time $ t^* = 100 $, and a weak solution with shock discontinuity persists for $ t > t^* $. However, with the inclusion of the dispersive term $ u_{xxt} $ in equation (\ref{eqn:EW}), the wave evolves into an undular bore, as seen in Fig.~\ref{Figure:E6}(b). These simulations demonstrate that the Compact6 method effectively captures the dispersive effects.
\par
The integrals $\mathfrak{I}_1, \mathfrak{I}_2$, and $\mathfrak{I}_3$ defined in equation (\ref{Invariants}), are no longer constant because the fluxes $\chi$ do not meet the conditions outlined in equation (\ref{Invariant:condition}). Table~\ref{Table:E6} shows the values of these integrals, along with the position and amplitude (equal to $\|u\|_\infty$) of the leading undulation at various time levels. The time-dependent variation of integrals $\mathfrak{I}_1, \mathfrak{I}_2$, and $\mathfrak{I}_3$ can be analytically computed and is expressed by
\begin{equation}
    \begin{split}
        M_1 &= \dfrac{d\mathfrak{I}_1}{dt} = \dfrac{d}{dt} \int_{-\infty}^{\infty} u \, dx = \dfrac{1}{2} u_0^2 = 5 \times 10^{-3},\\
        M_2 &= \dfrac{d\mathfrak{I}_2}{dt} = \dfrac{d}{dt} \int_{-\infty}^{\infty} (u^2 + \delta u_x^2) \, dx = \dfrac{2}{3} u_0^3 = 6.6667 \times 10^{-4},\\
        M_3 &= \dfrac{d\mathfrak{I}_3}{dt} = \dfrac{d}{dt} \int_{-\infty}^{\infty} u^3 \, dx = \dfrac{3}{4} u_0^4 = 7.5 \times 10^{-5}.
    \end{split}
\end{equation}
The results obtained from the Compact6 method for the time rate of change of the integrals are in close agreement with the analytical values. The calculated values are $M_1 = 4.9994 \times 10^{-3}$, $M_2 = 6.6658 \times 10^{-4}$, and $M_3 = 7.4985 \times 10^{-5}$.
\begin{table}[htbp!]
\centering
\captionof{table}{Development of an undular bore $\delta =1$, $u_0=0.1$, $d = 5$, $-20\leq x \leq 80$, $N=200$: $\mathfrak{I}_1, \mathfrak{I}_2, \mathfrak{I}_3$ and the position $x$ and amplitude $\|u\|_\infty$ of the leading undulation.}
\setlength{\tabcolsep}{0pt}
\begin{tabular*}{\textwidth}{@{\extracolsep{\fill}} l *{8}{c} }
\toprule
Time & $\mathfrak{I}_1$ & $\mathfrak{I}_2$ & $\mathfrak{I}_3$ & $x$ & $\|u\|_\infty$ \\
\midrule
 150& 2.749873 & 0.284106 & 0.027498 & 4.00 & 0.123733\\
 200& 2.999882 & 0.317460 & 0.031252 & 6.50 & 0.142425\\
 300& 3.499897 & 0.384223 & 0.038767 & 12.00 & 0.165481\\
 400& 3.999897 & 0.451066 & 0.046296 & 17.50 & 0.170815\\
 600& 4.999896 & 0.585010 & 0.061398 & 29.50 & 0.178813\\
 800& 5.999897 & 0.719309 & 0.076557 & 41.50 & 0.181478\\
\bottomrule
\end{tabular*}\label{Table:E6}
\end{table}
\begin{figure}[htbp!]
  \begin{minipage}[b]{0.45\linewidth}
    \centering
    \includegraphics[trim=0cm 0cm 0cm 0cm, clip=true,width=\linewidth]{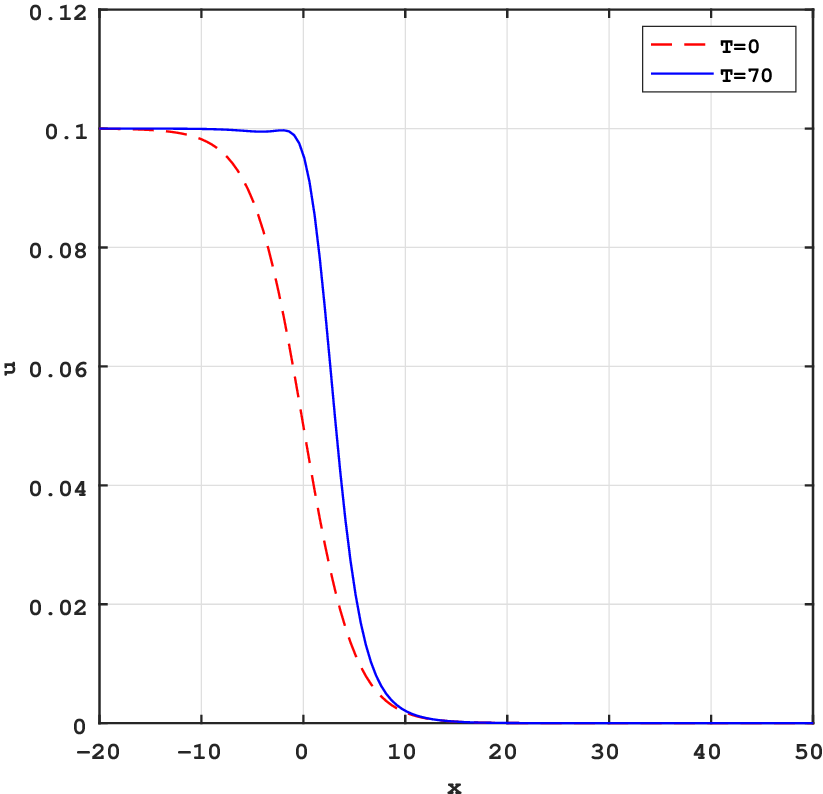}
    \subcaption*{\normalsize{\centering (a) $T = 70$ }}
  \end{minipage}\hfill
    \begin{minipage}[b]{0.45\linewidth}
    \centering
    \includegraphics[trim=0cm 0cm 0cm 0cm, clip=true,width=\linewidth]{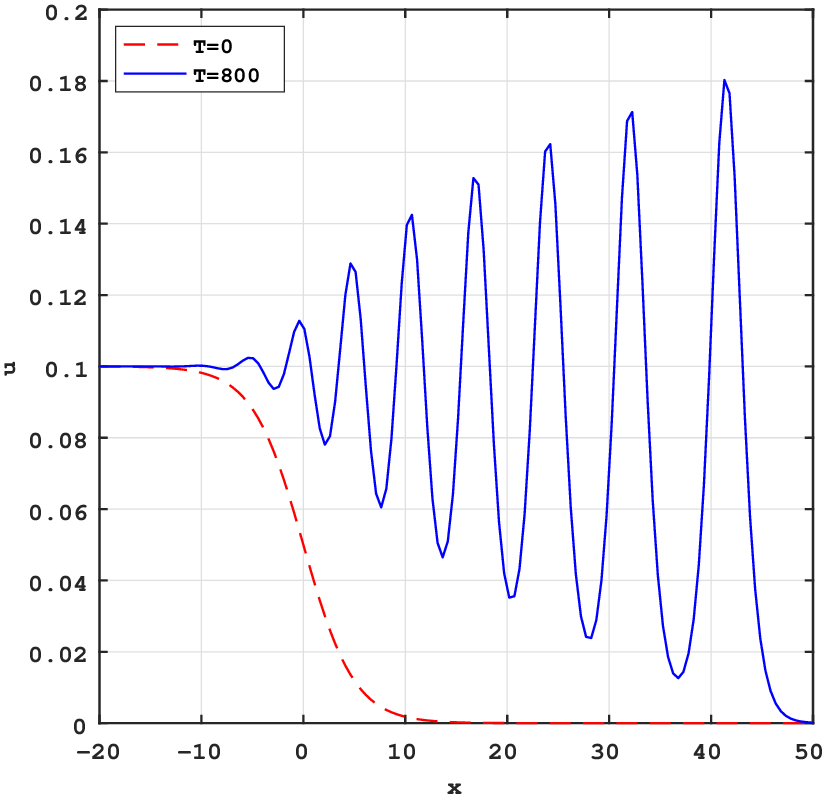}
    \subcaption*{\normalsize{\centering (b) $T = 800$}}
  \end{minipage}\hfill
  \caption{Numerical solution of EW equation with $N=200$ for Example \ref{example:5}.}\label{Figure:E6}
\end{figure}

\end{example}

\begin{example}\label{example:6}
\normalfont
Here, we examine the inhomogeneous BBMB equation (\ref{eqn:BBMB}), using the initial condition
\begin{equation}
    u(x,0) = u(x,0) = sech(x),  \quad -10<x<10,
\end{equation}
and the boundary conditions
\begin{equation}
    u(a,t) = sech(a-t), \quad u(b,t) = sech(b-t), \quad 0<t<T.
\end{equation}
The forcing term $g$ is chosen as
\begin{equation}
    \mathrm{g}(x,t) = \biggl(1 - 6 tanh^3(x-t) -2 tanh^2(x-t) - (sech(x-t)-5) tanh(x-t)\biggr) sech(x-t).
\end{equation}
following the formulation proposed by Dehghan et al.\cite{dehghan2014numerical}, which yields the exact solution $u(x,t) = sech(x-t)$. Using the Compact6 method, we approximate the solution at $ T = 1 $ with the interval $[-10, 10]$ and the specified initial and boundary conditions. The error analysis, presented in Table~\ref{Table:E7}, demonstrates the accuracy of our results in $ L^1 $, $ L^2 $, and $ L^{\infty} $ norms, showing that the method achieves sixth-order accuracy. 
Fig.~\ref{Figure:E7a} illustrates the comparison of $L^{\infty}$ errors at the final time \( T = 1 \), using a fixed time step size \( \tau = 10^{-4} \). The Compact6 scheme consistently yields lower errors than the CBSQI and ICBSQI methods, reinforcing its high precision and suitability for accurately solving the problem.
%
\begin{table}[htbp!]
\centering
\captionof{table}{Errors and order of convergence for Example \ref{example:6}.}
\setlength{\tabcolsep}{0pt}
\begin{tabular*}{\textwidth}{@{\extracolsep{\fill}} l *{8}{c} }
\toprule
 \textbf{N} & $\boldsymbol{L^{\infty}}$\textbf{-error} & \textbf{Rate} & $\boldsymbol{L^{1}}$\textbf{-error} & \textbf{Rate} & $\boldsymbol{L^{2}}$\textbf{-error} & \textbf{Rate}\\
\midrule
40& 5.0752e-03 & - & 8.3908e-04 & - & 1.4121e-03 & - \\
80& 7.7037e-05 & 6.0418 & 1.1634e-05 & 6.1724& 2.0708e-05 & 6.0915 \\
160& 1.1704e-06 & 6.0404 & 1.7234e-07 & 6.0769& 3.1358e-07 & 6.0452 \\
 320 & 1.7996e-08 & 6.0232 & 2.6243e-09 & 6.0372& 4.8375e-09 & 6.0184 \\
\bottomrule
\end{tabular*}\label{Table:E7}
\end{table}

\begin{figure}[htbp!]  
    \centering
    \begin{minipage}[b]{0.45\linewidth}
      \includegraphics[width=\linewidth]{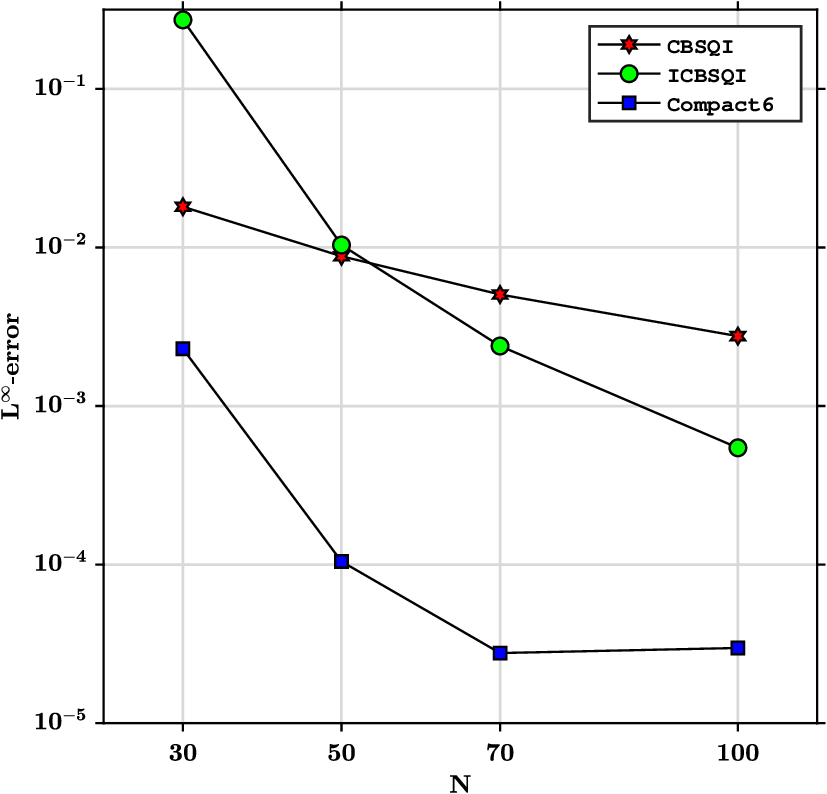}
    \end{minipage}\hfill
    \caption{ Comparison of CBSQI, ICBSQI and Compact6 schemes in terms of $L^{\infty}$ errors (in $log10$ scale) for Example~\ref{example:6} at $T=1$ and $\tau =0.0001$.}
    \label{Figure:E7a} 
\end{figure}

\end{example}

\section{Conclusion}\label{sec:5}
In this study, we developed a sixth-order compact finite difference scheme, Compact6, aimed at achieving high-accuracy numerical solutions for Sobolev-type equations. We began by outlining the Compact6 scheme for approximating both first and second-order derivatives, detailing the boundary treatments essential for practical implementation. The temporal derivatives are approximated using the explicit forward Euler difference method. The stability of the Compact6 method was rigorously analyzed in terms of  $L_2 $-stability, for the linear case. Using von Neumann stability analysis, we derived conditions under which the scheme remains stable and examined the amplification factor,  $\mathcal{C}(\theta)$, to ensure its decay properties. To validate the accuracy and effectiveness of the Compact6 scheme, we conducted extensive numerical experiments. The tests included cases of 1D and \textcolor{blue}{2D} advection-free flow, advection-diffusion flow, and applications involving the equal width equation—such as single solitary wave propagation, interactions of two and three solitary waves, and the formation of undular bores. Additionally, tests were conducted on the Benjamin–Bona–Mahony–Burgers equation.
\par
A comprehensive comparison of the $L^{\infty}$, $L^1$, and $L^2$ error norms and the associated convergence rates demonstrates that the proposed Compact6 scheme achieves approximately first-order temporal accuracy when coupled with the forward Euler method. The numerical results also consistently show that Compact6 significantly outperforms the existing CBSQI and ICBSQI schemes in terms of accuracy across various test problems and spatial resolutions. This superiority is evident not only in short-time simulations but also in long-time integrations, thereby affirming the robustness, stability, and high precision of the Compact6 approach for solving a wide range of Sobolev-type equations.\\
\newline
\noindent{\bf Acknowledgements }
The authors thank the unknown referee(s) for their careful reading of the manuscript and valuable suggestions that significantly improved the quality of the paper.
The author Rathan Samala is supported by NBHM, DAE, India (Ref. No. 02011/46/2021 NBHM(R.P.)/R \& D II/14874).\\
\newline
\noindent{\bf Declaration of competing interest }
The authors declare that they have no known competing financial interests or personal relationships that could have appeared to influence the work reported in this paper.\\
\newline
\noindent{\bf Data availability}
No data was used for the research described in the article.
\bibliographystyle{abbrv}
\bibliography{references}
\end{document}